\documentclass[11pt,reqno]{amsart}
\usepackage{amsfonts}
\usepackage{mathrsfs}
\usepackage{color}
\usepackage{cite}
\usepackage{amscd}
\usepackage{latexsym}
\usepackage{amsfonts}
\usepackage{graphicx}
\usepackage{CJK,indentfirst,amsmath,amsfonts,amssymb,amsthm,cite,cases,subeqnarray,setspace}

\begin{document}
\begin{CJK*}{GBK}{song}

\title[KdV limit of the Euler-Poisson system]
{KdV limit of the Euler-Poisson system}
\author[Y. Guo and X. Pu]
{Yan Guo\ \ \ \ \  Xueke Pu}  

\address{Yan Guo \newline
Division of Applied Mathematics, Brown University, Providence, RI 02912, USA.} \email{guoy@cfm.brown.edu}

\address{Xueke Pu \newline
Department of Mathematics, Chongqing University, Chongqing 400044, P.R.China;\newline
Division of Applied Mathematics, Brown University, Providence, RI 02912, USA.} \email{ puxueke@gmail.com}

\thanks{}

\begin{abstract}
Consider the scaling $\varepsilon^{1/2}(x-Vt)\to x,\ \varepsilon^{3/2}t\to t$ in the Euler-Poisson system for ion-acoustic waves \eqref{equ1}. We establish that as $\varepsilon\to 0$, the solutions to such Euler-Poisson system converge globally in time to the solutions of the Korteweg-de Vries equation.
\end{abstract}

\maketitle \numberwithin{equation}{section}
\newtheorem{proposition}{Proposition}[section]
\newtheorem{theorem}{Theorem}[section]
\newtheorem{lemma}[theorem]{Lemma}
\newtheorem{remark}[theorem]{Remark}
\newtheorem{hypothesis}[theorem]{Hypothesis}
\newtheorem{definition}{Definition}[section]
\newtheorem{corollary}{Corollary}[section]
\newtheorem{assumption}{Assumption}[section]

\section{Introduction}
\setcounter{section}{1}\setcounter{equation}{0}
The Euler-Poisson system is an important two-fluid model for describing the dymamics of a plasma. Consider the one dimensional Euler-Poisson system for ion-acoustic waves
\begin{equation}\label{equ1}
\begin{split}
&\partial_tn+\partial_x(nu)=0\\
&\partial_tu+u\partial_x u+\frac{1}{M}\frac{T_i\partial_x n}{n}=-\frac{e}{M}\partial_x\phi\\
&\partial^2_x\phi=4\pi e(\bar{n}e^{e\phi/T_e}-n),
\end{split}
\end{equation}
where $n(t,x),u(t,x)$ and $\phi(t,x)$ are the density, velocity of the ions and the electric potential at time $t\geq 0$, position $x\in\Bbb R$ respectively. The parameters $e>0$ is the electron charge, $T_e$ is the temperature of the electron, $M$ and $T_i$ are the mass and temperature of the ions respectively. The electrons are described by the so called isothermal Boltzmann relation
$$n_e=\bar{n}e^{e\phi/T_e},$$
where $\bar{n}$ is the equilibrium densities of the electrons.

Both experimental and theoretic studies show that in the long-wavelength limit, Korteweg-de Vries equation would govern the dynamics of \eqref{equ1}. However, only formal derivations of such KdV limit are known \cite{WT66,SG69,Guo95,HS02}. In this paper, we close this gap by justifying this limit rigourously.

\subsection{Formal KdV expansion}
By the classical Gardner-Morikawa transformation \cite{SG69}
\begin{equation}\label{equ39}
\varepsilon^{1/2}(x-Vt)\to x,\ \ \ \ \varepsilon^{3/2}t\to t,
\end{equation}
in \eqref{equ1}, we obtain the parameterized equation
\begin{equation}\label{equ10}
\begin{split}
&\varepsilon\partial_tn-V\partial_xn+\partial_x(nu)=0,\\
&\varepsilon\partial_tu-V\partial_xu+u\partial_xu+\frac{T_i}{M}\frac{\partial_xn}{n} =-\frac{e}{M}\partial_x\phi,\\
&\varepsilon\partial_x^2\phi=4\pi e(\bar{n}e^{e\phi/T_e}-n),
\end{split}
\end{equation}
where $\varepsilon$ is the amplitude of the initial disturbance and is assumed to be small compared with unity and $V$ is a velocity parameter to be determined. We consider the following formal expansion
\begin{equation}\label{expan-formal}
\begin{split}
n&=\bar{n}(1+\varepsilon^{1} n^{(1)}+\varepsilon^{2} n^{(2)}+\varepsilon^{3} n^{(3)}+\varepsilon^{4} n^{(4)}+\cdots)\\
u&=\varepsilon^{1} u^{(1)}+\varepsilon^{2} u^{(2)}+\varepsilon^{3}u^{(3)}+\varepsilon^{4} u^{(4)}+\cdots\\
\phi&=\varepsilon^{1}\phi^{(1)}+\varepsilon^{2}\phi^{(2)} +\varepsilon^{3}\phi^{(3)}+\varepsilon^{4} \phi^{(4)}+\cdots.
\end{split}
\end{equation}
Plugging \eqref{expan-formal} into \eqref{equ1}, we get a power series of $\varepsilon$, whose coefficients depend on $(n^{(k)},u^{(k)},\phi^{(k)})$ for $k=1,2,\cdots$.

\bigskip

\emph{The coefficients of $\varepsilon^0$}: The coefficient of $\varepsilon^0$ is automatically 0.

\bigskip

\emph{The coefficients of $\varepsilon$}: Setting the coefficient of $\varepsilon$ to be 0, we obtain
\begin{subequations}\label{e1}
\begin{numcases}{(\mathcal S_0)}
-V\partial_xn^{(1)}+\partial_xu^{(1)}=0\\
-V\partial_xu^{(1)}+\frac{T_i}{M}\partial_xn^{(1)}=-\frac{e}{M}\partial_x\phi^{(1)}\\
0=\frac{e}{T_e}\phi^{(1)}-n^{(1)}.
\end{numcases}
\end{subequations}
Write this equation in the matrix form
\begin{equation}\label{constraint}
\left[
\begin{array}{ccc}
V  &  -1  & 0  \\
{T_i}/{M}   &  -V  &  e/M\\
1  &  0  &  -e/T_e
\end{array}
\right]
\left[
\begin{array}{ccc}
\partial_xn^{(1)}\\
\partial_xu^{(1)}\\
\partial_x\phi^{(1)}
\end{array}
\right]=\left[
\begin{array}{ccc}
0\\
0\\
0
\end{array}
\right].
\end{equation}
In order to get a nontrivial solution for $n^{(1)},u^{(1)}$ and $\phi^{(1)}$, we require the determinant of the coefficient matrix to vanish so that
\begin{equation}\label{equ7}
\frac{T_i+T_e}{M}={V^2}.
\end{equation}
That is to say, we can (and need only to) adjust the velocity $V$, which is independent of any physical parameters, to derive the KdV equation. Furthermore, \eqref{e1} enables us to assume the relation
\begin{subequations}\label{equ6}
\begin{numcases}{(\mathcal L_1)}
u^{(1)}=Vn^{(1)},\label{equ6-1}\\
\phi^{(1)}=\frac{T_e}{e}n^{(1)},\label{equ6-2}
\end{numcases}
\end{subequations}
which makes \eqref{e1} valid. Only $n^{(1)}$ still needs to be determined.
\bigskip

\emph{The coefficients of $\varepsilon^2$ and the KdV equation for $n^{(1)}$}: Setting the coefficient of $\varepsilon^2$ to be 0, we obtain
\begin{subequations}\label{e2}
\begin{numcases}{(\mathcal S_1)}
\partial_tn^{(1)}-V\partial_xn^{(2)}+\partial_xu^{(2)} +\partial_x(n^{(1)}u^{(1)})=0, \label{e2-1}\\
\partial_tu^{(1)}-V\partial_xu^{(2)}+u^{(1)}\partial_xu^{(1)}\nonumber\\
\ \ \ \ \ \ \ \ \ \ \ \ +\frac{T_i}{M}(\partial_xn^{(2)}-n^{(1)}\partial_xn^{(1)}) =-\frac{e}{M}\partial_x\phi^{(2)},\label{e2-2}\\
\partial_x^2\phi^{(1)}=4\pi e\bar{n}(\frac{e}{T_e}\phi^{(2)}+\frac12((\frac{e}{T_e})\phi^{(1)})^2-n^{(2)}). \label{e2-3}
\end{numcases}
\end{subequations}
Differentiating \eqref{e2-3} with respect to $x$, multiplying \eqref{e2-1} with $V$, and \eqref{e2-3} with ${T_e}/{(4\pi e\bar{n}M)}$ respectively, and then adding them to \eqref{e2-2} together, we deduce that $n^{(1)}$ satisfies the Korteweg-de Vries equation
\begin{equation}\label{kdv}
\partial_tn^{(1)}+Vn^{(1)}\partial_xn^{(1)}+\frac{1}{2}\frac{T_e}{4\pi \bar{n}eMV}\partial_x^3n^{(1)}=0,
\end{equation}
where we have used the relation \eqref{equ6} and \eqref{equ7}, under which all the coefficients of $n^{(2)},u^{(2)}$ and $\phi^{(2)}$ vanish. We also note that the system \eqref{kdv} and \eqref{equ6} for $(n^{(1)},u^{(1)},\phi^{(1)})$ are self contained, which do not depend on $(n^{(j)},u^{(j)},\phi^{(j)})$ for $j\geq 2$. The above formal derivation for the case $T_i=0$ can be found in \cite{SG69}; while the derivation for the case $T_i>0$ is new.

Now we want to find out the equations satisfied by $(n^{(2)},u^{(2)},\phi^{(2)})$ assuming that $(n^{(1)},u^{(1)},\phi^{(1)})$ is known (solved form \eqref{kdv} and \eqref{equ6}). From \eqref{e2}, we can express $(n^{(2)},u^{(2)},\phi^{(2)})$ in terms of $(n^{(1)},u^{(1)},\phi^{(1)})$:
\begin{subequations}\label{equ16}
\begin{numcases}{(\mathcal L_2)}
\phi^{(2)}=\frac{T_e}{e}(n^{(2)}+h^{(1)}),\ \ h^{(1)}=\frac{1}{4\pi e\bar{n}}\partial_x^2\phi^{(1)}-\frac12{(\frac{e}{T_e}\phi^{(1)})^2},\label{equ16-1}\\
u^{(2)}=Vn^{(2)}+{g}^{(1)},\ \ \ \ \ \  {g}^{(1)}=\int^{x}\mathfrak{g}^{(1)}(t,\xi)d\xi,\label{equ16-2}\\
\ \ \ \ \ \ \ \ \ \ \ \ \ \ \ \ \ \ \ \ \ \ \ \ \ \ \ \ \ \ \ \ \mathfrak{g}^{(1)}=-\partial_tn^{(1)}+\partial_x(n^{(1)}u^{(1)}),\nonumber
\end{numcases}
\end{subequations}
which make \eqref{e2} valid. Only $n^{(2)}$ needs to be determined now.
\bigskip

\emph{The coefficients of $\varepsilon^3$ and the linearized KdV equation for $n^{(2)}$}: Setting the coefficient of $\varepsilon^3$ to be zero, we obtain
\begin{subequations}\label{e3}
\begin{numcases}{(\mathcal S_2)}
\partial_tn^{(2)}-V\partial_xn^{(3)}+\partial_xu^{(3)} +\partial_x(n^{(1)}u^{(2)}+n^{(2)}u^{(1)})=0,\label{e3-1}\\
\partial_tu^{(2)}-V\partial_xu^{(3)}+\partial_x(u^{(1)}u^{(2)}) +\frac{T_i}{M}\partial_xn^{(3)}\nonumber \\ \ \ \ \ \ \ \ \ \ \ \ \ -\frac{T_i}{M}[\partial_x(n^{(1)}n^{(2)}) -(n^{(1)})^2\partial_xn^{(1)}]=-\frac{e}{M}\partial_x\phi^{(3)},\label{e3-2}\\
\partial_x^2\phi^{(2)}=4\pi e\bar{n}[\frac{e}{T_e}\phi^{(3)}+(\frac{e}{T_e})^2\phi^{(1)}\phi^{(2)} +\frac1{3!}(\frac{e}{T_e}\phi^{(1)})^3-n^{(3)}].\label{e3-3}
\end{numcases}
\end{subequations}
Differentiating \eqref{e3-3} with respect to $x$, multiplying \eqref{e3-1} with $V$, and \eqref{e3-3} with ${T_e}/{(4\pi e\bar{n}M)}$ respectively, and then adding them to \eqref{e3-2} together, we deduce that $n^{(2)}$ satisfies the linearized inhomogeneous Korteweg-de Vries equation
\begin{equation}\label{e4}
\begin{split}
\partial_tn^{(2)}+V\partial_x(n^{(1)}n^{(2)})+\frac{1}{2}\frac{T_e}{4\pi \bar{n}eMV}\partial_x^3n^{(2)}=G^{(1)},
\end{split}
\end{equation}
where we have used \eqref{equ16} and $G^{(1)}=G^{(1)}(n^{(1)})$ depends only on $n^{(1)}$. Again, the system \eqref{e4} and \eqref{equ16} for $(n^{(2)},u^{(2)},\phi^{(2)})$ are self contained and do not depend on $(n^{(j)},u^{(j)},\phi^{(j)})$ for $j\geq 3$.

\bigskip

\emph{The coefficients of $\varepsilon^{k+1}$ and the linearized KdV equation for $n^{(k)}$}: Let $k\geq3$ be an integer. Recalling that in the $k^{th}$ step, by setting the coefficient of $\varepsilon^k$ to be 0, we obtain an evolution system $(\mathcal S_{k-1})$ for $(n^{(k-1)},u^{(k-1)},\phi^{(k-1)})$, from which we obtain
\begin{subequations}\label{relation}
\begin{numcases}{(\mathcal L_k)}
\phi^{(k)}=\frac{T_e}{e}(n^{(k)}+h^{(k-1)}),\ for\ some\ h^{(k-1)}\ depending \label{relation-1}\\ \ \ \ \ \ \ \ \ \ only\ on\ (n^{j},u^{j},\phi^{j})\ for\ 1\leq j\leq k-1,\nonumber\\
u^{(k)}=Vn^{(k)}+{g}^{(k-1)},\ {g}^{(k-1)}=\int^{x}\mathfrak{g}^{(k-1)}(t,\xi)d\xi,\ for\ some \label{relation-2}\\ \ \ \ \ \ \ \ \ \ \mathfrak{g}^{(k-1)}\ depending\ only\ on\ (n^{j},u^{j},\phi^{j})\ for\ 1\leq j\leq k-1.\nonumber
\end{numcases}
\end{subequations}
This relation makes $(\mathcal S_{k-1})$ valid, and we need only to determine $n^{(k)}$. By setting the coefficient of $\varepsilon^{k+1}$ to be 0, we obtain an evolution system $(\mathcal S_k)$ for $(n^{(k)},u^{(k)},\phi^{(k)})$. By the same procedure that leads to \eqref{e4}, we obtain the linearized inhomogeneous Korteweg-de Vries equation for $n^{(k)}$
\begin{equation}\label{linearized}
\begin{split}
\partial_tn^{(k)}+V\partial_x(n^{(1)}n^{(k)})+\frac{1}{2}\frac{T_e}{4\pi \bar{n}eMV}\partial_x^3n^{(k)}=G^{(k-1)},\ \ k\geq3,
\end{split}
\end{equation}
where $G^{(k-1)}$ depends only on $n^{(1)},\cdots,n^{k-1}$, which are ``known" from the first $(k-1)^{th}$ steps. Again, the system \eqref{relation} and \eqref{linearized} for $(n^{(k)},u^{(k)},\phi^{(k)})$ are self contained, which do not depend on $(n^{(j)},u^{(j)},\phi^{(j)})$ for $j\geq k+1$.

For the solvability of $(n^{(k)},u^{(k)},\phi^{(k)})$ for $k\geq 1$, we have the following two theorems.
\begin{theorem}\label{thm2}
Let $\tilde s_1\geq2$ be a sufficiently large integer. Then for any given initial data $n^{(1)}_0\in H^{\tilde s_1}(\Bbb R)$, there exists $\tau_*>0$ such that the initial value problem \eqref{kdv} and \eqref{equ6} has a unique solution
$$(n^{(1)},u^{(1)},\phi^{(1)})\in L^{\infty}(-\tau_*,\tau_*; H^{\tilde s_1}(\Bbb R))$$
with initial data $(n^{(1)}_0,Vn^{(1)}_0,T_en^{(1)}_0/e)$. Furthermore, by using the conservation laws of the KdV equation, we can extend the solution to any time interval $[-\tau,\tau]$.
\end{theorem}
This result is classical for the KdV equation, see for example \cite{KPV93}. See also \cite{Miura76,KdV} for more details on KdV equation.
\begin{theorem}\label{thm3}
Let $k\geq 2$ and $\tilde s_k\leq \tilde s_1-3(k-1)$ be a sufficiently large integer. Then for any $\tau>0$ and any given initial data $(n^{(k)}_0,u^{(k)}_0,\phi^{(k)}_0)\in H^{\tilde s_k}(\Bbb R)$,  the initial value problem \eqref{linearized} and \eqref{relation} with initial data $(n^{(k)}_0,u^{(k)}_0,\phi^{(k)}_0)$ satisfying \eqref{relation} has a unique solution
$$(n^{(k)},u^{(k)},\phi^{(k)})\in L^{\infty}(-\tau,\tau; H^{\tilde s_k}(\Bbb R)).$$
\end{theorem}

The proof of Theorem \ref{thm3} is standard. See Appendix. In the following, we will assume that these solutions $(n^{(k)},u^{(k)},\phi^{(k)})$ for $1\leq k\leq 4$ are sufficiently smooth. The optimality of $\tilde s_k$ will not be addressed in this paper.

\subsection{Main result}\label{sect-rem}
To show that $n^{(1)}$ converges to a solution of the KdV equation as $\varepsilon\to 0$, we must make the above procedure rigorous. Let $(n,u,\phi)$ be a solution of the scaled system \eqref{equ10} of the following expansion
\begin{equation}\label{expan}
\begin{split}
n&=\bar{n}(1+\varepsilon^{1} n^{(1)}+\varepsilon^{2} n^{(2)}+\varepsilon^{3} n^{(3)}+\varepsilon^{4} n^{(4)}+\varepsilon^3n^{\varepsilon}_R)\\
u&=\varepsilon^{1} u^{(1)}+\varepsilon^{2} u^{(2)}+\varepsilon^{3}u^{(3)}+\varepsilon^{4} u^{(4)}+\varepsilon^3u^{\varepsilon}_R\\
\phi&=\varepsilon^{1}\phi^{(1)}+\varepsilon^{2}\phi^{(2)} +\varepsilon^{3}\phi^{(3)}+\varepsilon^{4} \phi^{(4)}+\varepsilon^3\phi^{\varepsilon}_R,
\end{split}
\end{equation}
where $(n^{(1)},u^{(1)},\phi^{(1)})$ satisfies \eqref{equ6} and \eqref{kdv}, $(n^{(k)},u^{(k)},\phi^{(k)})$ satisfies \eqref{relation} and \eqref{linearized} for $2\leq k\leq 4$, and $(n^{\varepsilon}_R,u^{\varepsilon}_R,\phi^{\varepsilon}_R)$ is the remainder.

In the following, we derive the remainder system satisfied by $(n_R^{\varepsilon},u_R^{\varepsilon},\phi_R^{\varepsilon})$. To simplify the expression, we denote
\begin{equation*}
\begin{split}
\tilde n=n^{(1)}+\varepsilon n^{(2)}+\varepsilon^2n^{(3)}+\varepsilon^3n^{(4)}, \ \ \ \tilde u=u^{(1)}+\varepsilon u^{(2)}+\varepsilon^2u^{(3)}+\varepsilon^3u^{(4)}.
\end{split}
\end{equation*}
After careful computations (see Appendix for details), we obtain the following remainder system for $(n^{\varepsilon}_R, u^{\varepsilon}_R, \phi^{\varepsilon}_R)$:
\begin{subequations}\label{rem}
\begin{numcases}{}
\ \ \partial_tn^{\varepsilon}_R-\frac{V-u}{\varepsilon}\partial_xn^{\varepsilon}_R +\frac{n}{\varepsilon}\partial_xu^{\varepsilon}_R +\partial_x\tilde nu^{\varepsilon}_R +\partial_x\tilde un^{\varepsilon}_R+\varepsilon{\mathcal R_1}=0,\label{rem-1}\\
\ \ \partial_tu^{\varepsilon}_R-\frac{V-u}{\varepsilon}\partial_xu^{\varepsilon}_R +\frac{1}{\varepsilon}\frac{T_i}{M}\partial_xn^{\varepsilon}_R -\frac{T_i}{M}(\frac{\tilde n+\varepsilon n^{\varepsilon}_R}{n})\partial_xn^{\varepsilon}_R\nonumber\\
\ \ \ \ \ \ \ \ \ \ \ \ +\partial_x\tilde uu^{\varepsilon}_R -\frac{T_i}{M}\frac{b}{n}n^{\varepsilon}_R
+\varepsilon\mathcal R_{2}=-\frac{1}{\varepsilon}\frac{e}{M}\partial_x\phi_R^{\varepsilon},\label{rem-2}\\
\varepsilon\partial_x^2\phi^{\varepsilon}_R=4\pi e\bar{n}[\frac{e}{T_e}\phi^{\varepsilon}_R+\varepsilon(\frac{e}{T_e})^2 \phi^{(1)}\phi^{\varepsilon}_R-n^{\varepsilon}_R]+\varepsilon^2{\mathcal R}_3,\label{rem-3}
\end{numcases}
\end{subequations}
where 
\begin{subequations}\label{rrr}
\begin{numcases}{}
\ \ b=\partial_xn^{(1)} +\varepsilon(\partial_xn^{(2)}-n^{(1)}\partial_xn^{(1)})\nonumber\\ \ \ \ \ \ \ \ \ +\varepsilon^2 (\partial_xn^{(3)}+[(n^{(1)})^2-\partial_xn^{(1)}] \partial_xn^{(1)}+n^{(1)}\partial_xn^{(2)})\nonumber\\
\ \ \ \ \ \ \ \ +\varepsilon^3(\partial_xn^{(4)} -[\partial_x(n^{(1)}n^{(3)}) +(n^{(2)}-(n^{(1)})^2)\partial_xn^{(2)}\nonumber\\
\ \ \ \ \ \ \ \ +((n^{(1)})^3-2n^{(1)}n^{(2)})\partial_xn^{(1)}]),\label{rrr-1}\\
{\mathcal R_1}=\partial_tn^{(4)}+\sum_{{1\leq i,j\leq 4; i+j\geq5}}\varepsilon^{i+j-5}\partial_x(n^{(i)}u^{(j)}),\label{rrr-2}\\
{\mathcal R_2}=\partial_tu^{(4)}+\sum_{{1\leq i,j\leq 4; i+j\geq5}}\varepsilon^{i+j-5}u^{(i)}\partial_xu^{(j)} +\frac{T_i}{M}\frac{1}{n}\big\{\text{finite }\nonumber\\
\ \ \ \ \ \ \ \  \text{combination of } n^{(i)}(1\leq i\leq 4)\text{ and their derivatives}\big\},\label{rrr-3}\\
{\mathcal R}_3=\left[\frac{1}{2}(\frac{e}{T_e})^2({\varepsilon}\phi^{\varepsilon}_R) +(\frac{e}{T_e})^2 (\phi^{(2)}+\frac{1}{2}\frac{e}{T_e}(\phi^{(1)})^2)\right]\phi^{\varepsilon}_R +\widehat R'({\varepsilon}\phi^{\varepsilon}_R).\label{rrr-4}
\end{numcases}
\end{subequations}
One can refer to the Appendix for the detailed derivation of $\mathcal R_3$, which is a smooth function of $\phi^{\varepsilon}_R$. In particular, $\mathcal R_3$ does not involve any derivatives of $\phi^{\varepsilon}_R$. The mathematical key difficulty is to derive estimates for the remainders $(n^{\varepsilon}_R, u^{\varepsilon}_R, \phi^{\varepsilon}_R)$ uniformly in $\varepsilon$.

Our main result of this paper is the following
\begin{theorem}\label{thm1}
Let $\tilde s_i\geq2$ in Theorem \ref{thm2} and \ref{thm3} be sufficiently large and $(n^{(1)},u^{(1)},\phi^{(1)})\in H^{\tilde s_1}$ be a solution constructed in Theorem \ref{thm2} for the KdV equation with initial data $(n^{(1)}_0,u^{(1)}_0,\phi^{(1)}_0)\in H^{\tilde s_1}$ satisfying \eqref{equ6}. Let $(n^{(i)},u^{(i)},\phi^{(i)})\in H^{\tilde s_i}$ $(i=2,3,4)$ be solutions of \eqref{linearized} and \eqref{relation} constructed in Theorem \ref{thm3} with initial data $(n^{(i)}_0,u^{(i)}_0,\phi^{(i)}_0)\in H^{\tilde s_i}$ satisfying \eqref{relation}. Let ${(n^{\varepsilon}_R}_0,{u^{\varepsilon}_R}_0,{\phi^{\varepsilon}_R}_0)$ satisfy \eqref{rem} and assume
\begin{equation*}
\begin{split}
n_0&=\bar{n}(1+\varepsilon^{1} n^{(1)}_0+\varepsilon^{2} n^{(2)}_0+\varepsilon^{3} n^{(3)}_0+\varepsilon^{4} n^{(4)}_0+\varepsilon^3{n^{\varepsilon}_R}_0),\\
u_0&=\varepsilon^{1} u^{(1)}_0+\varepsilon^{2} u^{(2)}_0+\varepsilon^{3}u^{(3)}_0+\varepsilon^{4} u^{(4)}_0+\varepsilon^3{u^{\varepsilon}_R}_0,\\
\phi_0&=\varepsilon^{1}\phi^{(1)}_0+\varepsilon^{2}\phi^{(2)}_0 +\varepsilon^{3}\phi^{(3)}_0+\varepsilon^{4} \phi^{(4)}_0+\varepsilon^3{\phi^{\varepsilon}_R}_0.
\end{split}
\end{equation*}
Then for any $\tau>0$, there exists $\varepsilon_0>0$ such that if $0<\varepsilon<\varepsilon_0$, the solution of the EP system \eqref{equ10} with initial data $(n_0,u_0,\phi_0)$ can be expressed as
\begin{equation*}
\begin{split}
n&=\bar{n}(1+\varepsilon^{1} n^{(1)}+\varepsilon^{2} n^{(2)}+\varepsilon^{3} n^{(3)}+\varepsilon^{4} n^{(4)}+\varepsilon^3n^{\varepsilon}_R),\\
u&=\varepsilon^{1} u^{(1)}+\varepsilon^{2} u^{(2)}+\varepsilon^{3}u^{(3)}+\varepsilon^{4} u^{(4)}+\varepsilon^3u^{\varepsilon}_R,\\
\phi&=\varepsilon^{1}\phi^{(1)}+\varepsilon^{2}\phi^{(2)} +\varepsilon^{3}\phi^{(3)}+\varepsilon^{4} \phi^{(4)}+\varepsilon^3\phi^{\varepsilon}_R,
\end{split}
\end{equation*}
such that for all $0<\varepsilon<\varepsilon_0$,
\par 1) when $T_i>0$,
\begin{equation*}
\sup_{[0,\tau]}\|(n^{\varepsilon}_R,u^{\varepsilon}_R, \phi^{\varepsilon}_R)\|_{H^{2}}^2\leq C_{\tau}\left(1+\|({n^{\varepsilon}_R}_0,{u^{\varepsilon}_R}_0, {\phi^{\varepsilon}_R}_0)\|_{H^{2}}^2\right),
\end{equation*}
\par 2) when $T_i=0$,
\begin{equation*}
\begin{split}
\sup_{[0,\tau]}&\{\|(n^{\varepsilon}_R,u^{\varepsilon}_R, \phi^{\varepsilon}_R)\|_{H^{2}}^2 +\varepsilon\|(\partial_x^3u^{\varepsilon}_R, \partial_x^3\phi^{\varepsilon}_R)\|_{L^{2}}^2 +\varepsilon^2\|\partial_x^4\phi^{\varepsilon}_R\|_{L^{2}}^2\}\\
&\leq C_{\tau}\left(1+\|({n^{\varepsilon}_R}_0,{u^{\varepsilon}_R}_0, {\phi^{\varepsilon}_R}_0)\|_{H^{2}}^2 +\varepsilon\|(\partial_x^3{u^{\varepsilon}_R}_0, \partial_x^3{\phi^{\varepsilon}_R}_0)\|_{L^{2}}^2 +\varepsilon^2\|\partial_x^4{\phi^{\varepsilon}_R}_0\|_{L^{2}}^2\right).
\end{split}
\end{equation*}
\end{theorem}

\begin{remark}
While we get a global uniform in $\varepsilon$ estimate for the $H^2$ norm of the remainders, the $H^3$ norm or the $H^4$ norm may blow up in finite time. However, they are both uniformly bounded after multiplied by $\varepsilon^{1/2}$ and $\varepsilon$ respectively.
\end{remark}

Our result provides a rigorous and unified justification of the KdV equation limit of the Euler-Poisson system for ion-acoustic waves with Boltamann relation. The classical formal derivation in \cite{SG69} deals with only the case of $T_i=0$, while our results cover all the case of $T_i\geq 0$. When $T_i>0$, the control of the remainder falls into the framework of Grenier \cite{Gre97}, where the author studied some singular limits by using the pseudo-differential operator (PsDO) techniques for singular perturbations of hyperbolic systems. But suitable decomposition of \eqref{rem-3} is required.

Unfortunately, in the classical case of $T_i=0$, we cannot apply Grenier's machinery to get uniform estimate for the remainders. This is because when $T_i=0$, the matrix $P_{\varepsilon}^{-1}$ given by \eqref{PInverse} is not a bounded family of PsDOs of order 0 any more,  see \cite{Gre97} or \cite{Stein93} for more details on PsDO theory. To overcome this difficulty, we need to employ a careful combination of delicate energy estimate together with analysis of the structure of the remainder system.

The basic plan is to first estimate some uniform bound for $(u^{\varepsilon}_R,\phi^{\varepsilon}_R)$ and then recover the estimate for $n^{\varepsilon}_R$ from the estimate of $\phi^{\varepsilon}_R$ by the Poisson equation \eqref{rem-3} (see Lemma \ref{L1}). We want to apply the Gronwall lemma to complete the proof. To state clearly, we first define (see \eqref{def-A})
\begin{equation}\label{triple}
|\!|\!|(u^{\varepsilon}_R,\phi^{\varepsilon}_R)|\!|\!|^2_{\varepsilon}=\|u^{\varepsilon}_R\|_{H^2}^2 +\|\phi^{\varepsilon}_R\|_{H^2}^2 +\varepsilon\|\partial_x^3u^{\varepsilon}_R\|^2 +\varepsilon\|\partial_x^3\phi^{\varepsilon}_R\|^2 +\varepsilon^2\|\partial_x^4\phi^{\varepsilon}_R\|^2.
\end{equation}
As we will see, the zeroth order, the first to the second order estimates for $(u^{\varepsilon}_R,\phi^{\varepsilon}_R)$ can be controlled in terms of $|\!|\!|(u^{\varepsilon}_R,\phi^{\varepsilon}_R)|\!|\!|^2_{\varepsilon}$. Unfortunately, the third order estimate of $\sqrt{\varepsilon}(u^{\varepsilon}_R,\phi^{\varepsilon}_R)$ involves a bad term $\mathcal B^{(3\times\varepsilon)}$ (see \eqref{bad3} and Remark \ref{remark})
\begin{equation}\label{equ41}
\begin{split}
\mathcal B^{(3\times\varepsilon)}=-\int\partial_{x}^{3}\phi^{\varepsilon}_R \partial_x\left[\frac{\varepsilon^2}n\right]\partial_{t} \partial_x^4\phi^{\varepsilon}_Rdx,
\end{split}
\end{equation}
where $\partial_{t} \partial_x^4\phi^{\varepsilon}_R$ cannot be controlled in terms of $|\!|\!|(u^{\varepsilon}_R,\phi^{\varepsilon}_R)|\!|\!|_{\varepsilon}$. Even worse, this difficulty persists no matter how high the Sobolev order or the expansion order is. For example, when we want to estimate the $H^k$ norm of $(u^{\varepsilon}_R,\phi^{\varepsilon}_R)$, we get a term $\mathcal B^{(k)}=-\int\partial_{x}^{k}\phi^{\varepsilon}_R \partial_x\left[\frac{\varepsilon}n\right]\partial_{t} \partial_x^{k+1}\phi^{\varepsilon}_Rdx$ with the same structure of \eqref{equ41}. Note that $\mathcal B^{(3\times\varepsilon)}$ in equation \eqref{bad3} is just $\mathcal B^{(3)}$ multiplied by $\varepsilon$.

Fortunately, we are able to employ the precise structure of
\eqref{rem-3} to overcome such a difficulty. In the second order estimate, we can extract a precise term $\mathcal B^{(2)}$ (after integration by parts, see \eqref{equ-B2})
\begin{equation*}
\begin{split}
\mathcal B^{(2)}=\int\partial_{x}^3\phi^{\varepsilon}_R \partial_x\left[\frac{\varepsilon}n\right] \partial_{t}\partial_x^2\phi^{\varepsilon}_R.
\end{split}
\end{equation*}
Even though $\partial_t\partial_x^4\phi^{\varepsilon}_R$ in \eqref{equ41} is out of control, the combination of
$$\partial_t\partial_x^2\phi^{\varepsilon}_R -\varepsilon\partial_t\partial_x^4\phi^{\varepsilon}_R$$
can be controlled in terms of $|\!|\!|(u^{\varepsilon}_R,\phi^{\varepsilon}_R)|\!|\!|_{\varepsilon}^2$ by using the Poisson equation \eqref{rem-3}.

In recent years, there have been a large number of studies of the Euler-Poisson (Maxwell) equation and related various singular limit \cite{CG00,ELT01,LT02,LT03,LMZ10,GP11,Te07,Guo98,GGP11,GJ10}. In \cite{SW00}, KdV equation is derived rigorously from the water-wave equation.

This paper is organized as follows. In Section 2, we prove the limit for the case of $T_i>0$, by using the PsDO framework of Grenier \cite{Gre97}. In Section 3, we prove the limit for the classical case of $T_i=0$, where more delicate estimate is required. Throughout this paper, $\|\cdot\|$ denotes the $L^2$ norm.

\section{Uniform energy estimates: the case $T_i>0$}\label{sect-energy}
\setcounter{section}{2}\setcounter{equation}{0}
In this section, we give the energy estimates uniformly in $\varepsilon$ for the case of $T_i>0$ for the remainder system \eqref{rem} of $(n^{\varepsilon}_R, u^{\varepsilon}_R, \phi^{\varepsilon}_R)$. This section is divided into two parts. In the first one, we introduce an abstract form of the remainder equations of $n^{\varepsilon}_R$ and $u^{\varepsilon}_R$, while $\phi^{\varepsilon}_R$ is only included implicitly. This abstract form is more suitable for us to apply the PsDO framework in Grenier \cite{Gre97}. Then in the second part, we establish energy estimates and prove the main theorem for $T_i>0$. We remark that PsDO framework is applicable mainly because this system is symmetrizable when $T_i>0$, see also \cite{Majda84}.

For notational convenience, we normalize the physical constants $e,M,T_i,T_e$ to be 1 and $\bar{n}=(4\pi e)^{-1}$ in \eqref{rem} throughout this section. Therefore, $V=\sqrt2$ by \eqref{constraint}. Let $\tau\geq 0$ be arbitrarily fixed, we will establish estimates in $L^{\infty}(0,\tau;H^{s'})$ for any $2\leq s'\leq \tilde s_4-3$, where $\tilde s_4$ is sufficiently large and fixed in Theorem \ref{thm2}.
\subsection{Reduction}
We follow Grenier's framework of \cite{Gre97} (see also \cite{CG00}). Before we give the uniform estimate, we first reduce \eqref{rem} into an abstract form.
\begin{lemma}
Let $(n^{\varepsilon}_R,u^{\varepsilon}_R,\phi^{\varepsilon}_R)$ be a solution to \eqref{rem} and $w=(n^{\varepsilon}_R,u^{\varepsilon}_R)^{T}$. Then $w$ satisfies the following system
\begin{equation}\label{equ11}
w_t+\mathcal A_{\varepsilon}(w)w+\mathcal R(w)=0,
\end{equation}
where $\mathcal A_{\varepsilon}$ is a family of pseudodifferential operators whose symbol depends on the solution $w$, and can be decomposed into the sum of a ``regular" part and a ``singular" part
\begin{equation*}
\mathcal A_{\varepsilon}(w)=\mathcal A_{1,\varepsilon}(w)+\mathcal A_{2,\varepsilon},
\end{equation*}
whose symbols are respectively the following matrices
\begin{equation}\label{equ12}
A_{1,\varepsilon}(w)=\left[
\begin{array}{ccc}
i\xi(\tilde u+\varepsilon^2u^{\varepsilon}_R)  &  i\xi(\tilde n+\varepsilon^2 n^{\varepsilon}_R)\\
-\frac{i\xi(\tilde n+\varepsilon^2 n^{\varepsilon}_R)}{n}-\frac{i\xi\phi^{(1)}}{{(1+\varepsilon\xi^2)}{(1+ \varepsilon\phi^{(1)}+\varepsilon\xi^2)}}  &  i\xi(\tilde u+\varepsilon^2 u^{\varepsilon}_R)
\end{array}
\right]
\end{equation}
and
\begin{equation}\label{equ13}
A_{2,\varepsilon}=\frac1\varepsilon \tilde A_{2,\varepsilon}(\xi)=\frac{1}{\varepsilon}\left[
\begin{array}{ccc}
-\sqrt{2}i\xi  &  i\xi\\
i\xi+\frac{i\xi}{1+\varepsilon\xi^2}  &  -\sqrt{2}i\xi
\end{array}
\right].
\end{equation}
In \eqref{equ11},
\begin{equation}\label{equ51}
\mathcal R(w)=\mathcal F_{\varepsilon}(w)w+\mathcal N(w),
\end{equation}where $\mathcal F_{\varepsilon}(w)$ is the coefficient matrix before $(n^{\varepsilon}_R,u^{\varepsilon}_R)^T$:
\begin{equation}\label{equ24}
\mathcal F_{\varepsilon}(w)=\left[
\begin{array}{ccc}
\partial_x\tilde u  &  \partial_x\tilde n\\
-\frac{b}{n}  &  \partial_x\tilde u
\end{array}
\right],
\end{equation}
and $\mathcal N=[\mathcal N_1, -\mathcal N_2]^T$ is defined by \eqref{er2}. There exists constant $C_{\alpha}$
\begin{equation}\label{equ23}
\|\mathcal F_{\varepsilon}(w)w\|_{H^{\alpha}}+\|\mathcal N(w)\|_{H^{\alpha}}\leq C_{\alpha}(1+\|n^{\varepsilon}_R\|+\|u^{\varepsilon}_R\|)
\end{equation}
for every $2\leq\alpha\leq s'$.
\end{lemma}
\begin{proof}
To reduce the system \eqref{rem} to an evolution system for $(n^{\varepsilon}_R,u^{\varepsilon}_R)$, we need to express $\phi^{\varepsilon}_R$ in terms of $n^{\varepsilon}_R$ and $u^{\varepsilon}_R$. We therefore consider the decomposition
\begin{equation}\label{equ44}
\phi^{\varepsilon}_R=\Phi_1+\Phi_2+\Phi_3,
\end{equation}
where $\Phi_1,\Phi_2$ and $\Phi_3$ are specified below.

Recall \eqref{rem-3},
\begin{equation}\label{equ40}
\varepsilon\partial_x^2\phi^{\varepsilon}_R =(1+\varepsilon\phi^{(1)})\phi^{\varepsilon}_R-n^{\varepsilon}_R +\varepsilon^2{\mathcal R}_3.
\end{equation}
First, we define
\begin{equation*}
\Phi_1=Op(\frac{1}{(1+\varepsilon\phi^{(1)})+\varepsilon\xi^2})n^{\varepsilon}_R,
\end{equation*}
where $Op(\frac{1}{(1+\varepsilon\phi^{(1)})+\varepsilon\xi^2})$ is a PsDO with limited smoothness (see \cite{Gre97} for more details), for all $0<\varepsilon<\varepsilon_1$ for some $\varepsilon_1>0$. We then have
\begin{equation}\label{equ47}
\|\Phi_1\|_{H^{\alpha}}\leq C\|n^{\varepsilon}_R\|_{H^{\alpha}}.
\end{equation}
In fact, by standard PsDO calculus, we have
\begin{equation*}
Op({(1+\varepsilon\phi^{(1)})+\varepsilon\xi^2}) Op(\frac{1}{(1+\varepsilon\phi^{(1)}) +\varepsilon\xi^2})n^{\varepsilon}_R=n^{\varepsilon}_R
+\varepsilon\tilde{\mathcal S}_1n^{\varepsilon}_R,
\end{equation*}
where $\tilde{\mathcal S}_1$ is a bounded operator from $H^{\alpha}$ to $H^{\alpha+1}$ defined by
\begin{equation*}
\begin{split}
\tilde{\mathcal S}_1=&\frac{1}{\varepsilon}\left(({1+\varepsilon\phi^{(1)}})Op(\frac{1}{1 +\varepsilon\phi^{(1)} +\varepsilon\xi^2})-Op(\frac{1+\varepsilon\phi^{(1)}} {1+\varepsilon\phi^{(1)} +\varepsilon\xi^2})\right)\\
=&\phi^{(1)} Op(\frac{1}{1+\varepsilon\phi^{(1)} +\varepsilon\xi^2})-Op(\frac{\phi^{(1)}} {1+\varepsilon\phi^{(1)}+\varepsilon\xi^2}).
\end{split}
\end{equation*}
We remark that the $\varepsilon$ in front of $\tilde{\mathcal S_1}$ is very important, since it cancels part of the singularity of $\varepsilon^{-1}\partial_x\phi_R^{\varepsilon}$ in \eqref{rem}. Equivalently, $\Phi_1$ is a solution of
\begin{equation}\label{equ48}
\varepsilon\partial_x^2\Phi_1=(1+\varepsilon\phi^{(1)})\Phi_1-n^{\varepsilon}_R -\varepsilon \tilde{\mathcal S}_1n^{\varepsilon}_R.
\end{equation}
This enables us to define $\Phi_2$ to be the solution of
\begin{equation}\label{equ49}
\varepsilon\partial_x^2\Phi_2=(1+\varepsilon\phi^{(1)})\Phi_2+\varepsilon \tilde{\mathcal S}_1n^{\varepsilon}_R.
\end{equation}
It is straightforward that for $\alpha\geq 1$
\begin{equation}\label{equ45}
\|\Phi_2\|_{H^{\alpha}}\leq \varepsilon C\|\tilde{\mathcal S}_1n^{\varepsilon}_R\|_{H^{\alpha}}\leq \varepsilon C\|n^{\varepsilon}_R\|_{H^{\alpha-1}}.
\end{equation}
Finally, we define $\Phi_3$ to be the solution of
\begin{equation}\label{equ42}
\varepsilon\partial_x^2\Phi_3=(1+\varepsilon\phi^{(1)})\Phi_3+\varepsilon^2{\mathcal R}_3.
\end{equation}
By superposition of linear equations \eqref{equ48}, \eqref{equ49} and \eqref{equ42}, we get \eqref{equ44}.

Now, we consider the decomposition of $-\frac{1}{\varepsilon}\partial_x\phi^{\varepsilon}_R$ on the RHS of \eqref{rem-2}. For the expression of $\mathcal R_3$ in \eqref{rrr-4}, by Lemma \ref{L8} there exists constant 
$C=C(\|\phi^{(i)}\|_{H^{\tilde s_i}}, \varepsilon\|\phi^{\varepsilon}_R\|_{H^{\alpha}})$ such that
\begin{equation}\label{equ43}
\begin{split}
\|\mathcal R_3\|_{H^{\alpha}}\leq & C\|\phi^{\varepsilon}_R\|_{H^{\alpha}}\\
\leq & C(\|\Phi_1\|_{H^{\alpha}},\|\Phi_2\|_{H^{\alpha}},\|\Phi_3\|_{H^{\alpha}}),
\end{split}
\end{equation}
for any $\alpha$ such that $2\leq \alpha\leq s'$, for some $s'\leq\tilde s_4$ depending on $\tilde s_4$. Taking inner product of \eqref{equ42} with $\partial_x^{\alpha}\Phi_3$ and integrating by parts, we have
\begin{equation*}
\begin{split}
\varepsilon\|\partial_x^{\alpha+1}\Phi_3\|^2+\int\partial_x^{\alpha} ((1+\varepsilon\phi^{(1)})\Phi_3)\partial_x^{\alpha}\Phi_3 \leq & \varepsilon^2C(\|\Phi_1\|_{H^{\alpha}},\|\Phi_2\|_{H^{\alpha}},\|\Phi_3\|_{H^{\alpha}}) \|\partial_x^{\alpha}\Phi_3\|.
\end{split}
\end{equation*}
On the other hand, since $\|\varepsilon\phi^{(1)}\|_{L^{\infty}}<1/2$ when $0<\varepsilon<\varepsilon_1$ for some $\varepsilon_1>0$, we obtain
\begin{equation}\label{equ46}
\begin{split}
\varepsilon&\|\partial_x\Phi_3\|_{H^{\alpha}}+\|\Phi_3\|_{H^{\alpha}}\leq \varepsilon^2C(\|\Phi_1\|_{H^{\alpha}},\|\Phi_2\|_{H^{\alpha}}).
\end{split}
\end{equation}
Therefore, from \eqref{equ45}, \eqref{equ46} and \eqref{equ47},
\begin{equation}\label{equ21}
\begin{split}
\|\frac{1}{\varepsilon}(\partial_x\Phi_2+\partial_x\Phi_3)\|_{H^{\alpha}}\leq & \|\frac{1}{\varepsilon}\partial_x\Phi_2\|_{H^{\alpha}} +\|\frac{1}{\varepsilon}\partial_x\Phi_3\|_{H^{\alpha}}\\
\leq & C\|n^{\varepsilon}_R\|_{H^{\alpha}}+C(\|\Phi_1\|_{H^{\alpha}},\|\Phi_2\|_{H^{\alpha}})\\
\leq & C(\|n^{\varepsilon}_R\|_{H^{\alpha}}).
\end{split}
\end{equation}

On the other hand, by symbolic calculus, we have
\begin{equation*}
\begin{split}
-\frac{1}{\varepsilon}\partial_x\Phi_1
=&-\frac{1}{\varepsilon}Op(i\xi)Op(\frac{1}{(1+\varepsilon\phi^{(1)}) +\varepsilon\xi^2})n^{\varepsilon}_R\\
=&-\frac{1}{\varepsilon}Op(\frac{i\xi}{(1+\varepsilon\phi^{(1)}) +\varepsilon\xi^2})n^{\varepsilon}_R+\mathcal S_2n^{\varepsilon}_R,
\end{split}
\end{equation*}
where
\begin{equation}\label{equ22}
\begin{split}
\mathcal S_2=Op(\frac{\partial_x\phi^{(1)}}{(1+\varepsilon \phi^{(1)}+\varepsilon\xi^2)^2})
\end{split}
\end{equation}
is a bounded operator from $H^{\alpha}$ to $H^{\alpha}$ for every $\alpha\leq s'$. Recalling \eqref{equ44}, we obtain the decomposition of $-\frac{1}{\varepsilon}\partial_x\phi^{\varepsilon}_R$ on the RHS of \eqref{rem-2}:
\begin{equation}
\begin{split}
-\frac{1}{\varepsilon}\partial_x\phi^{\varepsilon}_R=-\frac{1}{\varepsilon}Op(\frac{i\xi}{(1+\varepsilon\phi^{(1)}) +\varepsilon\xi^2})n^{\varepsilon}_R+\mathcal S_2n^{\varepsilon}_R -\frac{1}{\varepsilon}(\partial_x\Phi_2+\partial_x\Phi_3).
\end{split}
\end{equation}

Defining
\begin{equation}\label{er2}
\begin{split}
\mathcal N_1=-\varepsilon\mathcal R_1,\ \ \ {\mathcal N}_2=\mathcal S_2n^{\varepsilon}_R-\frac{1}{\varepsilon}(\partial_x\Phi_2+\partial_x\Phi_3) -\varepsilon\mathcal R_{2},
\end{split}
\end{equation}
where $\mathcal R_1$ and $\mathcal R_2$ are defined in \eqref{rem-2} and \eqref{rem-3} respectively, we can transform the remainder system \eqref{rem} into the abstract form \eqref{equ11}. Note also that from \eqref{equ21} and \eqref{equ22}, $\mathcal N$ is bounded by \eqref{equ23}.
\end{proof}

\subsection{Energy estimates}
In this subsection, we will complete the proof of Theorem \ref{thm1} for the case $T_i>0$. For this, we need only uniform energy estimates for \eqref{equ11}, where the matrices $A_{1,\varepsilon}$ and $A_{1,\varepsilon}$ are given by \eqref{equ12} and \eqref{equ13} respectively. To further simplify the notations, we denote
\begin{equation}\label{equ50}
\begin{split}
N_{R}=&\tilde n+\varepsilon^2n^{\varepsilon}_R;\ \ \ \ \ \ \ \ \  \ U_{R}=\tilde u+\varepsilon^2u^{\varepsilon}_R,\\
n_1=&n=1+\varepsilon N_{R},\ \ \ \ \ n_2=1+\varepsilon\phi^{(1)}+\varepsilon\xi^2.
\end{split}
\end{equation}
In these notations,
\begin{equation*}
A_{\varepsilon}=i\xi\left[
\begin{array}{ccc}
(U_R-\frac{1}{\varepsilon})  &  \frac{n_1}{\varepsilon} \\
\frac1{\varepsilon(1+\varepsilon N_R)}+\frac{1}{\varepsilon n_2}  & (U_R-\frac{1}{\varepsilon})
\end{array}
\right],
\end{equation*}
whose eigenvalues are
\begin{equation*}
\lambda_{\pm}=i\xi((U_R-\frac1{\varepsilon})\pm \frac{\sqrt{n_1}}{\varepsilon}\frac{\sqrt{n_2+n_1}}{\sqrt{n_1n_2}})
\end{equation*}
and their normalized eigenvectors are
\begin{equation*}
e_{\pm}
=\left[
\begin{array}{ccc}
\frac{n_1\sqrt{n_2}}{\sqrt{n_1^2n_2+n_2+n_1}}  \\
\pm\frac{\sqrt{n_2+n_1}}{\sqrt{n_1^2n_2+n_2+n_1}}
\end{array}
\right].
\end{equation*}
Let
\begin{equation*}
P_{\varepsilon}=\left[
\begin{array}{ccc}
\frac{n_1\sqrt{n_2}}{\sqrt{n_1^2n_2+n_2+n_1}} &  \frac{n_1\sqrt{n_2}}{\sqrt{n_1^2n_2+n_2+n_1}} \\
\frac{\sqrt{n_2+n_1}}{\sqrt{n_1^2n_2+n_2+n_1}}  &
-\frac{\sqrt{n_2+n_1}}{\sqrt{n_1^2n_2+n_2+n_1}}
\end{array}
\right],
\end{equation*}
\begin{equation}\label{PInverse}
P_{\varepsilon}^{-1}=\frac12\left[
\begin{array}{ccc}
\frac{\sqrt{n_1^2n_2+n_2+n_1}}{n_1\sqrt{n_2}}  &  \frac{\sqrt{n_1^2n_2+n_2+n_1}}{\sqrt{n_2+n_1}}\\
\frac{\sqrt{n_1^2n_2+n_2+n_1}}{n_1\sqrt{n_2}}  & -\frac{\sqrt{n_1^2n_2+n_2+n_1}}{\sqrt{n_2+n_1}}
\end{array}
\right],
\end{equation}
and
\begin{equation*}
B_\varepsilon=\left[\begin{array}{ccc}
\lambda_+  &  0 \\
0 & \lambda_- \end{array}
\right],
\end{equation*}
we have the decomposition
\begin{equation}\label{diag}
A_{\varepsilon}=P_{\varepsilon}B_{\varepsilon}P_{\varepsilon}^{-1}.
\end{equation}

Now, we are ready to prove Theorem \ref{thm1} for the case $T_i>0$.
\begin{proof}[\textbf{Proof of Theorem \ref{thm1} for $T_i>0$}]
We prove this theorem by energy estimates. First, we note that for every $\varepsilon>0$, \eqref{equ11} has smooth solutions in some time interval $[0,T_{\varepsilon}]$ dependent on $\varepsilon$. Let $\mathcal C=Op(P^{-1}_{\varepsilon})$, and define the norm
$$|||w(t)|||_s^2\equiv\sum_{|\alpha|\leq s}\|\mathcal C\partial_x^{\alpha}w(t)\|^2.$$
We will bound $\partial_{t}|||w|||_{s'}^2$ for $\alpha\leq s'$. By a direct computtaion, we have
\begin{equation}\label{equ15}
\begin{split}
\partial_t\|\mathcal C\partial_x^{\alpha}w\|_{L^2}^2=&2\Re((\partial_t\mathcal C)\partial_x^{\alpha}w,\mathcal C\partial_x^{\alpha}w)-2\Re(\mathcal C[\partial_x^{\alpha},\mathcal A]w,\mathcal C\partial_x^{\alpha}w)\\
&-2\Re(\mathcal C\mathcal A\partial_x^{\alpha}w,\mathcal C\partial_x^{\alpha}w)-2\Re(\mathcal C\partial_x^{\alpha}\mathcal R,\mathcal C\partial_x^{\alpha}w)\\
=&:I+II+III+IV.
\end{split}
\end{equation}

\emph{Estimate of $I$.} Since $\mathcal C$ is a bounded family of matrix-valued PsDO of order 0, it is a uniformly bounded operator from $L^2\to L^2$. On the other hand,
\begin{equation*}
\begin{split}
\partial_t\mathcal C=&Op(\partial_t P_{\varepsilon}^{-1})=\sum_i\varepsilon^i\partial_{n^{(i)}}P_{\varepsilon}^{-1}\partial_tn^{(i)} +\varepsilon^3\partial_{n^{\varepsilon}_R}P_{\varepsilon}^{-1}\partial_tn^{\varepsilon}_R.
\end{split}
\end{equation*}
From \eqref{equ11} and the expressions for $n_1$ and $n_2$ in \eqref{equ50}, we have
\begin{equation*}
\begin{split}
\|\varepsilon\partial_tn^{\varepsilon}_R\|_{H^{s'-1}}\leq C(\|(n^{(i)},u^{(i)}, \phi^{(i)})\|_{H^{\tilde s_i}}, \|(n^{\varepsilon}_R, u^{\varepsilon}_R)\|_{H^{s'}})
\end{split}
\end{equation*}
and since $n^{(i)}$ are the first four known profiles, we have
\begin{equation*}
\begin{split}
\|\partial_tn^{(i)}\|_{H^{s'-1}}\leq C(\|n^{(i)}\|_{H^{\tilde s_i}},\|u^{(i)}\|_{H^{\tilde s_i}},\|\phi^{(i)}\|_{H^{\tilde s_i}})
\end{split}
\end{equation*}
for $i=1,2,3,4$. Therefore,
\begin{equation*}
\begin{split}
\|\partial_t\mathcal C\|_{H^{s'-1}}\leq C, \ \ \ \ s'>\frac d2+1,
\end{split}
\end{equation*}
for some $C=C(\varepsilon\|n^{\varepsilon}_R\|_{H^{s'}}, \varepsilon\|u^{\varepsilon}_R\|_{H^{s'}})$. In other words, $\partial_t\mathcal C$ is a uniformly bounded operator from $L^2$ to $L^2$. Consequently,
\begin{equation}\label{equ25}
\begin{split}
|I|\leq C_1\|\partial_x^{s'}w\|^2.
\end{split}
\end{equation}

\emph{Estimate of $II$ in \eqref{equ15}.} By the definition of and $\mathcal A_{2,\varepsilon}$, we know that
\begin{equation*}
[\partial_x^{\alpha},\mathcal A_{2,\varepsilon}]=0.
\end{equation*}
Since $\mathcal A_{1,\varepsilon}$ is a PsDO of order 1, by the commutator estimates that \cite{Majda84}, we have
\begin{equation*}
\begin{split}
\|[\partial_x^{\alpha},\mathcal A_{1,\varepsilon}]w\|_{L^2}\leq C(\|(n^{(i)},u^{(i)})\|_{H^{\tilde s_i}},\varepsilon\|(n^{\varepsilon}_R, u^{\varepsilon}_R)\|_{H^{s'}})\|w\|_{H^{\alpha}},
\end{split}
\end{equation*}
so that
\begin{equation}\label{equ28}
|II|\leq C(\|(n^{(i)},u^{(i)})\|_{H^{\tilde s_i}},\varepsilon\|(n^{\varepsilon}_R, u^{\varepsilon}_R)\|_{H^{s'}})\|w\|_{H^{\alpha}}^2.
\end{equation}

\emph{Estimate of $III$ in \eqref{equ15}.} Using the diagonalization \eqref{diag}, we split
\begin{equation}\label{equ14}
\begin{split}
(\mathcal C\mathcal A&\partial_x^{\alpha}w,\mathcal C\partial_x^{\alpha}w)\\
=&(\mathcal C\mathcal A\partial_x^{\alpha}w,\mathcal C\partial_x^{\alpha}w) -(Op(B_{\varepsilon}P_{\varepsilon}^{-1})\partial_x^{\alpha}w,\mathcal C\partial_x^{\alpha}w)\\
&+(Op(B_{\varepsilon}P_{\varepsilon}^{-1})\partial_x^{\alpha}w,\mathcal C\partial_x^{\alpha}w) -(Op(B_{\varepsilon})Op(P_{\varepsilon}^{-1})\partial_x^{\alpha}w,\mathcal C\partial_x^{\alpha}w)\\
&+(Op(B_{\varepsilon})Op(P_{\varepsilon}^{-1})\partial_x^{\alpha}w,\mathcal C\partial_x^{\alpha}w)\\
=&III_1+III_2+III_3.
\end{split}
\end{equation}
Let us first consider the term $III_1$. Since $A_{\varepsilon}$ depends on $n^{(i)}$, $n^{\varepsilon}_R$ in the form of $\varepsilon^in^{(i)}$, $\varepsilon^3 n^{\varepsilon}_R$ for $i=1,2,3,4$, $D_{n^{(i)}}A_{\varepsilon}$ and $D_{n^{\varepsilon}_R}A_{\varepsilon}$ are all bounded families of symbols of order 1. Furthermore, $P_{\varepsilon}^{-1}$ is a uniformly bounded family of symbols of order 0, and we have
\begin{equation*}
\begin{split}
\|\mathcal C\mathcal A-Op(B_{\varepsilon}P_{\varepsilon}^{-1})\|_{L^2\to L^2}\leq C(\|n^{(i)}\|_{H^{\tilde s_i}},\varepsilon\|n^{\varepsilon}_R\|_{H^{s'}}).
\end{split}
\end{equation*}
Similarly, since $D_{\xi}^{\alpha}B_{\varepsilon}\nabla_v^{\alpha}P_{\varepsilon}^{-1}$ are bounded symbols of order $1-\alpha$ for $III_2$, we have
\begin{equation*}
\begin{split}
\|Op(B_{\varepsilon})Op(P_{\varepsilon}^{-1}) -Op(B_{\varepsilon}P_{\varepsilon}^{-1})\|_{L^2\to L^2}\leq C(\|n^{(i)}\|_{H^{\tilde s_i}},\varepsilon\|n^{\varepsilon}_R\|_{H^{s'}}).
\end{split}
\end{equation*}
Finally, We consider $III_3$. Since $\lambda_{\pm}$ are purely imaginary when $\varepsilon<\varepsilon_2$ is sufficiently small, and $B_{\varepsilon}=diag[\lambda_{+},\lambda_{-}]$ is diagonal, $B_{\varepsilon}^*=-B_{\varepsilon}$. Therefore, by using the properties of the adjoint operator (symbolic calculus), we have $B_{\varepsilon}^*\in S^1$ and
\begin{equation*}
\begin{split}
Op(B_{\varepsilon})^*\sim \sum_{\alpha}\frac{1}{\alpha!}\partial_{\xi}^{\alpha}D_x^{\alpha}\bar B_{\varepsilon}(x,\xi).
\end{split}
\end{equation*}
On the other hand, since $B_{\varepsilon}$ depends on $n^{(i)}$ and $n^{\varepsilon}_R$ through $\varepsilon^in^{(i)}$ and $\varepsilon^3n^{\varepsilon}_R$, there exists a bounded operator $\tilde B_{\varepsilon}$ from $L^2\to L^2$ such that
\begin{equation*}
\begin{split}
\tilde B_{\varepsilon}=Op(B_{\varepsilon})+Op(B_{\varepsilon})^*
\end{split}
\end{equation*}
with bound
\begin{equation*}
\begin{split}
\|\tilde B_{\varepsilon}\|_{L^2\to L^2}\leq C(\|n^{(i)}\|_{H^{\tilde s_i}},\varepsilon\|n^{\varepsilon}_R\|_{H^{s'}}).
\end{split}
\end{equation*}
Consequently, we obtain from \eqref{equ14}
\begin{equation}\label{equ26}
|III|\leq C(\|n^{(i)}\|_{H^{\tilde s_i}},\varepsilon\|n^{\varepsilon}_R\|_{H^{s'}}) (\|n^{\varepsilon}_R\|^2_{H^{s'}}+\|u^{\varepsilon}_R\|^2_{H^{s'}}).
\end{equation}

\emph{Estimate of $IV$ in \eqref{equ15}.} Recall $\mathcal R(w)=\mathcal F_{\varepsilon}(w)w+\mathcal N(w)$ in \eqref{equ51}. Since $\mathcal R$ is a nonlinear bounded operator, from \eqref{equ23} we have for every $\alpha\geq2$,
\begin{equation*}
\|\mathcal R(w)\|_{H^{\alpha}}\leq C_{\alpha}(\|n^{\varepsilon}_R\|_{H^{s'}}+\|u^{\varepsilon}_R\|_{H^{s'}})
\end{equation*}
for some constant
\begin{equation*}
\begin{split}
C_\alpha=C_\alpha(\|(n^{(i)},u^{(i)},\phi^{(i)})\|_{H^{\tilde s_i}}, \varepsilon\|n^{\varepsilon}_R\|_{H^{s'}},\varepsilon\|u^{\varepsilon}_R\|_{H^{s'}}).
\end{split}
\end{equation*}
Since $\mathcal C\in S^0$ uniformly in $\varepsilon$, we obtain
\begin{equation}\label{equ27}
\|IV\|_{s'}\leq C(1+\|n^{\varepsilon}_R\|^2_{H^{s'}}+\|u^{\varepsilon}_R\|^2_{H^{s'}}).
\end{equation}
Therefore, from \eqref{equ15}, \eqref{equ28}, \eqref{equ26} and \eqref{equ27} we obtain
\begin{equation*}
\partial_t|||w|||_{s'}^2\leq C(\|n^{(i)}\|_{H^{\tilde s_i}},\varepsilon\|n^{\varepsilon}_R\|_{H^{s'}})(1+\|w\|^2_{H^{s'}}).
\end{equation*}

We claim that $\|\cdot\|_{H^{s'}}$ and $|||\cdot|||_{s'}$ are equivalent. Since $\mathcal C$ is a bounded family symbols of $S^0$, we have
\begin{equation*}
\|\mathcal C\partial_x^{\alpha}w\|_{L^2}^2\leq C(\|n^{(i)}\|_{H^{\tilde s_i}}, \varepsilon\|n^{\varepsilon}_R\|_{H^{s'}})\|\partial_x^{\alpha}w\|_{L^2}^2, \ \ \alpha\leq s'
\end{equation*}
and hence
\begin{equation*}
|||w|||_{s'}^2\leq C(\|n^{(i)}\|_{H^{\tilde s_i}}, \varepsilon\|n^{\varepsilon}_R\|_{H^{s'}})\|w\|_{H^{s'}}^2.
\end{equation*}
On the other hand, since $P_{\varepsilon}^{-1}$ and $P_{\varepsilon}$ depend on $\xi$, $n^{(i)}$, $n^{\varepsilon}_R$ through $\sqrt{\varepsilon}\xi$, $\varepsilon^in^{(i)}$, $\varepsilon^3n^{\varepsilon}_R$, we therefore have
\begin{equation*}
\begin{split}
Op(P_\varepsilon)Op(P^{-1}_{\varepsilon})=I+\varepsilon^{3/2}\mathcal P,
\end{split}
\end{equation*}
for some $L^2\to L^2$ bounded operator $\mathcal P$. Hence,
\begin{equation*}
\begin{split}
\|\partial_{x}^{\alpha}w\|_{L^2}^2\leq & \|Op(P_\varepsilon)Op(P^{-1}_{\varepsilon})\partial_{x}^{\alpha}w\|_{L^2}^2 +\varepsilon^{3}C(\|n^{(i)}\|_{H^{\tilde s_i}}, \varepsilon\|n^{\varepsilon}_R\|_{H^{s'}})\|\partial_{x}^{\alpha}w\|_{L^2}^2.
\end{split}
\end{equation*}
By the $L^2$-boundedness of $Op(P_{\varepsilon})$, when $\varepsilon$ is sufficiently small we have
\begin{equation*}
\|\partial_{x}^{\alpha}w\|_{L^{2}}^2\leq 2C(\|n^{(i)}\|_{H^{\tilde s_i}}, \varepsilon\|n^{\varepsilon}_R\|_{H^{s'}})
\|Op(P^{-1}_{\varepsilon})\partial_{x}^{\alpha}w\|_{L^2}^2.
\end{equation*}
Summation over $|\alpha|\leq s'$ yields the equivalence between $\|\cdot\|_{H^{s'}}$ and $|||\cdot|||_{s'}$.

Therefore, we finally obtain the estimate of the form
\begin{equation*}
\partial_t|||w|||_{s'}^2\leq C(\|n^{(i)}\|_{H^{s'}},\varepsilon\|n^{\varepsilon}_R\|_{H^{s'}}) (1+|||w|||_{s'}^2).
\end{equation*}
Since $C$ depends on $\|n^{\varepsilon}_R\|_{H^{s'}}$ through $\varepsilon\|n^{\varepsilon}_R\|_{H^{s'}}$, we obtain an existence time $T_{\varepsilon}\geq \tau$ for any $\tau>0$ uniformly in $\varepsilon$. From the decomposition of $\phi^{\varepsilon}_R$ in \eqref{equ44}, we recover the uniform in $\varepsilon$ estimate for $\|\phi^{\varepsilon}_R\|_{H^{s'}}$.

The proof of Theorem \ref{thm1} for the case $T_i>0$ is then complete for $s'=2$. We indeed have proved a stronger result that holds for any $s'\geq2$ integers.
\end{proof}

\section{Uniform energy estimates: the case $T_i=0$}
\setcounter{section}{3}\setcounter{equation}{0}
In the cold plasma $(T_i=0)$ case, the procedure in Section 2 is not applicable for two main reasons: the system cannot be symmetrized and $P_{\varepsilon}^{-1}$ is not a PsDO of order $0$. In this section, we handle this case, which requires a combination of energy method and analysis of remainder equation \eqref{rem}.

Throughout this section, we set $T_i=0$ and renormalize all the other constants to be $1$. Hence $V=1$, and from \eqref{rem} we obtain the following remainder equation
\begin{subequations}\label{rem2}
\begin{numcases}{}
\partial_tn^{\varepsilon}_R-\frac{1-u}{\varepsilon}\partial_xn^{\varepsilon}_R +\frac{n}{\varepsilon}\partial_xu^{\varepsilon}_R+\partial_x\tilde nu^{\varepsilon}_R+\partial_x\tilde un^{\varepsilon}_R +\varepsilon{\mathcal R_1}=0 \label{rem2-1}\\
\partial_tu^{\varepsilon}_R-\frac{1-u}{\varepsilon}\partial_xu^{\varepsilon}_R +\partial_x\tilde uu^{\varepsilon}_R+\varepsilon\mathcal R_{2}=-\frac{1}{\varepsilon}\partial_x\phi_R^{\varepsilon}\label{rem2-2}\\
\varepsilon\partial_x^2\phi^{\varepsilon}_R=(\phi^{\varepsilon}_R +\varepsilon \phi^{(1)}\phi^{\varepsilon}_R-n^{\varepsilon}_R)+\varepsilon^2{\mathcal R}_3,\label{rem2-3}
\end{numcases}
\end{subequations}
where 
$\mathcal R_1, \mathcal R_2$ and $\mathcal R_3$ are given by \eqref{rrr} with $T_i=0$. In particular, $\mathcal R_1$ and $\mathcal R_2$ depend only on $(n^{(i)},u^{(i)})$ and $\mathcal R_3$ does not involve any derivatives of $\phi^{\varepsilon}_R$.

In the following, we will give uniform estimates of system \eqref{rem2}. To simplify the proof slightly, we will assume that \eqref{rem2} has smooth solutions in very small time $\tau_{\varepsilon}>0$ dependent on $\varepsilon>0$. Recall that
\begin{equation}\label{def-A}
\begin{split}
|\!|\!|(u^{\varepsilon}_R,\phi^{\varepsilon}_R)|\!|\!|^2_{\varepsilon}=\|u^{\varepsilon}_R\|_{H^2}^2 +\|\phi^{\varepsilon}_R\|_{H^2}^2 +\varepsilon\|\partial_x^3u^{\varepsilon}_R\|^2 +\varepsilon\|\partial_x^3\phi^{\varepsilon}_R\|^2 +\varepsilon^2\|\partial_x^4\phi^{\varepsilon}_R\|^2.
\end{split}
\end{equation}
Let $\tilde C$ be a constant independent of $\varepsilon$, which will be determined later, much larger than the bound $|\!|\!|(u^{\varepsilon}_R,\phi^{\varepsilon}_R)(0)|\!|\!|^2_{\varepsilon}$ of the initial data. It is classical that there exists $\tau_{\varepsilon}>0$ such that on $[0,\tau_{\varepsilon}]$,
\begin{equation*}
\begin{split}
\|n^{\varepsilon}_R\|_{H^2}^2,\ \ |\!|\!|(u^{\varepsilon}_R,\phi^{\varepsilon}_R)(t)|\!|\!|^2_{\varepsilon} \leq \tilde C.
\end{split}
\end{equation*}
As a direct corollary, there exists some $\varepsilon_1>0$ such that $n$ is bounded from above and below $1/2<n<3/2$ and $u$ is bounded by $|u|<1/2$ when $\varepsilon<\varepsilon_1$. Since $\mathcal R_3$ is a smooth function of $\phi^{\varepsilon}_R$ (see Appendix), there exists some constant $C_1=C_1({\varepsilon}\tilde C)$ for any $\alpha,\beta\geq0$ such that
\begin{equation}\label{equ17}
\begin{split}
|\partial^{\alpha}_{\phi^{(i)}}\partial^{\beta}_{\phi^{\varepsilon}_R}\mathcal R_3|\leq C_1=C_1({\varepsilon}\tilde C),
\end{split}
\end{equation}
where $C_1(\cdot)$ can be chose to be nondecreasing in its argument.

We will show that for any given $\tau>0$, there is some $\varepsilon_0>0$, such that the existence time $\tau_{\varepsilon}>\tau$ for any $0<\varepsilon<\varepsilon_0$. We first prove the following Lemma \ref{L1}-\ref{L3}, in which we bound $n^{\varepsilon}_R$ and $\partial_t\phi^{\varepsilon}_R$ in terms of $\phi^{\varepsilon}_R$.
\begin{lemma}\label{L1}
Let $(n^{\varepsilon}_R,u^{\varepsilon}_R,\phi^{\varepsilon}_R)$ be a solution to \eqref{rem2} and $\alpha\geq 0$ be an integer. There exist some constants $0<\varepsilon_1<1$ and $C_1=C_1(\varepsilon\tilde C)$ such that for every $0<\varepsilon<\varepsilon_1$,
\begin{equation}\label{equ31}
\begin{split}
C_1^{-1}\|\partial_x^{\alpha}n^{\varepsilon}_R\|^2\leq & \|\partial_x^{\alpha}\phi^{\varepsilon}_R\|^2 +\varepsilon\|\partial_x^{\alpha+1}\phi^{\varepsilon}_R\|^2 +\varepsilon^2\|\partial_x^{\alpha+2}\phi^{\varepsilon}_R\|^2\leq C_1\|\partial_x^{\alpha}n^{\varepsilon}_R\|^2.
\end{split}
\end{equation}
\end{lemma}
\begin{proof}
When $\alpha=0$, taking inner product of \eqref{rem2-3} with $\phi^{\varepsilon}_R$, we have
\begin{equation}\label{equ30}
\begin{split}
\|\phi^{\varepsilon}_R\|^2+\varepsilon\|\partial_x\phi^{\varepsilon}_R\|^2=&\int n^{\varepsilon}_R\phi^{\varepsilon}_R -\int\varepsilon\phi^{(1)}|\phi^{\varepsilon}_R|^2-\int\varepsilon^2\mathcal R_3\phi^{\varepsilon}_R.
\end{split}
\end{equation}
From \eqref{equ29} in the Appendix, we have
\begin{equation*}
\begin{split}
\|\mathcal R_3\|_{L^2}\leq C_1(\varepsilon\tilde C)\|\phi^{\varepsilon}_R\|.
\end{split}
\end{equation*}
When $\varepsilon<\varepsilon_1$ is sufficiently small, $C_1(\varepsilon\tilde C)\leq C_1(1)$ is a fixed constant, and therefore
\begin{equation}\label{equ52}
\begin{split}
|\int\varepsilon^2\mathcal R_3\phi^{\varepsilon}_R|\leq \frac18\|\phi^{\varepsilon}_R\|^2.
\end{split}
\end{equation}
Since $\phi^{(1)}$ is known and is bounded in $L^{\infty}$, there exists some $0<\varepsilon_1<1$ such that for $0<\varepsilon<\varepsilon_1$,
\begin{equation}\label{equ53}
\begin{split}
\left|\int\varepsilon\phi^{(1)}|\phi^{\varepsilon}_R|^2\right|\leq & \frac18\|\phi^{\varepsilon}_R\|^2.
\end{split}
\end{equation}
By applying the H\"older's inequality to the first term on the RHS of \eqref{equ30}, we have
\begin{equation}\label{equ54}
\begin{split}
|\int n^{\varepsilon}_R\phi^{\varepsilon}_R| \leq \frac14\|\phi^{\varepsilon}_R\|^2+\|n^{\varepsilon}_R\|^2.
\end{split}
\end{equation}
By \eqref{equ30}-\eqref{equ54},
\begin{equation*}
\begin{split}
\|\phi^{\varepsilon}_R\|^2+\varepsilon\|\partial_x\phi^{\varepsilon}_R\|^2 \leq& \frac12\|\phi^{\varepsilon}_R\|^2+\|n^{\varepsilon}_R\|^2.
\end{split}
\end{equation*}
Hence, we have shown that there exists some $\varepsilon_1>0$ such that for $0<\varepsilon<\varepsilon_1$,
\begin{equation}\label{equ55}
\begin{split}
\|\phi^{\varepsilon}_R\|^2+\varepsilon\|\partial_x\phi^{\varepsilon}_R\|^2\leq & 2\|n^{\varepsilon}_R\|^2.
\end{split}
\end{equation}
Taking inner product with $\varepsilon\partial_x^2\phi^{\varepsilon}_R$ and integration by parts, we have similarly
\begin{equation}\label{equ56}
\begin{split}
\varepsilon\|\partial_x\phi^{\varepsilon}_R\|^2+\varepsilon^2\|\partial_x^2\phi^{\varepsilon}_R\|^2\leq & 2\|n^{\varepsilon}_R\|^2.
\end{split}
\end{equation}
On the other hand, from the equation \eqref{rem2-3}, there exist some $C$ such that
\begin{equation}\label{equ57}
\begin{split}
\|n^{\varepsilon}_R\|^2\leq & \|\phi^{\varepsilon}_R\|^2+\varepsilon^2\|\partial_x^2\phi^{\varepsilon}_R\|^2 +C\varepsilon^2\|\phi^{\varepsilon}_R\|^2+{(C_1(1))}^2\|\phi^{\varepsilon}_R\|^2\\
\leq & C(\|\phi^{\varepsilon}_R\|^2+\varepsilon^2\|\partial_x^2\phi^{\varepsilon}_R\|^2).
\end{split}
\end{equation}
Putting \eqref{equ55}-\eqref{equ57} together, we deduce the inequality \eqref{equ31} for $\alpha=0$.

For higher order inequalities, we differentiate the Poisson equation \eqref{rem2-3} with $\partial_x^{\alpha}$ and then take inner product with $\partial_x^{\alpha}\phi^{\varepsilon}_R$ and $\varepsilon\partial_x^{\alpha+2}\phi^{\varepsilon}_R$ separately. The Lemma follows by the same procedure of the case $\alpha=0$.
\end{proof}

Recall $|\!|\!|(u^{\varepsilon}_R,\phi^{\varepsilon}_R)|\!|\!|_{\varepsilon}$ in \eqref{def-A}. We remark that only $\|n^{\varepsilon}_R\|_{H^2}$ can be bounded in terms of $|\!|\!|(u^{\varepsilon}_R,\phi^{\varepsilon}_R)|\!|\!|_{\varepsilon}$ and no higher order derivatives of $n^{\varepsilon}_R$ is allowed in Lemma \ref{L1}. This is one of the reasons that why the estimate in the section is delicate.

\begin{lemma}\label{L2}
Let $(n^{\varepsilon}_R,u^{\varepsilon}_R,\phi^{\varepsilon}_R)$ be a solution to \eqref{rem2}. There exist some constant $C$ and $C_1=C_1(\varepsilon\tilde C)$, such that
\begin{equation}\label{equ32}
\begin{split}
\|\varepsilon\partial_tn^{\varepsilon}_R\|^2\leq & C(\|\phi^{\varepsilon}_R\|_{H^1}^2+\|u^{\varepsilon}_R\|_{H^1}^2 +\varepsilon\|\partial_x^2\phi^{\varepsilon}_R\|^2 +\varepsilon^2\|\partial_x^{3}\phi^{\varepsilon}_R\|^2)+C\varepsilon.
\end{split}
\end{equation}
\begin{equation}\label{equ33}
\begin{split}
\|\varepsilon\partial_{tx}n^{\varepsilon}_R\|^2\leq & C_1(\|u^{\varepsilon}_R\|_{H^2}^2+\|\phi^{\varepsilon}_R\|_{H^2}^2 +\varepsilon\|\partial_x^3\phi^{\varepsilon}_R\|^2 +\varepsilon^2\|\partial_x^{4}\phi^{\varepsilon}_R\|^2)+C\varepsilon.
\end{split}
\end{equation}
By \eqref{def-A}, it is useful to rewrite in the form
\begin{equation*}
\begin{split}
\|\varepsilon\partial_{t}n^{\varepsilon}_R\|_{H^1}^2\leq & C_1|\!|\!|(u^{\varepsilon}_R,\phi^{\varepsilon}_R)|\!|\!|^2_{\varepsilon}+C\varepsilon.
\end{split}
\end{equation*}
\end{lemma}
\begin{proof}
From \eqref{rem2-1}, we have
\begin{equation*}
\begin{split}
\varepsilon\partial_tn^{\varepsilon}_R=(1-u)\partial_xn^{\varepsilon}_R -n\partial_xu^{\varepsilon}_R -\varepsilon\partial_x\tilde un^{\varepsilon}_R -\varepsilon\partial_x\tilde nu^{\varepsilon}_R -\varepsilon^2\mathcal R_1.
\end{split}
\end{equation*}
Since $1/2<n<3/2$ and $|u|<1/2$, taking $L^2$-norm yields
\begin{equation*}
\begin{split}
\|\varepsilon\partial_tn^{\varepsilon}_R\|^2\leq&\|(1-u)\partial_xn^{\varepsilon}_R\|^2 +\|n\partial_xu^{\varepsilon}_R\|^2 +\varepsilon^2\|\partial_x\tilde un^{\varepsilon}_R\|^2+ \varepsilon^2\|\partial_x\tilde nu^{\varepsilon}_R\|^2 +\varepsilon^4\|\mathcal R_1\|^2\\
\leq & C(\|\partial_xn^{\varepsilon}_R\|^2+\|\partial_xu^{\varepsilon}_R\|^2) +C\varepsilon^2(\varepsilon^2+\|n^{\varepsilon}_R\|^2+\|u^{\varepsilon}_R\|^2).
\end{split}
\end{equation*}
Applying \eqref{equ31} with $\alpha=1$, we deduce \eqref{equ32}.

To prove \eqref{equ33}, we take $\partial_x$ of \eqref{rem2-1} to obtain
\begin{equation*}
\begin{split}
\|\varepsilon\partial_{tx}n^{\varepsilon}_R\|^2\leq C(\|u^{\varepsilon}_R\|_{H^2}^2+\|n^{\varepsilon}_R\|_{H^2}^2) +C\varepsilon^6\int|\partial_xu^{\varepsilon}_R|^2|\partial_xn^{\varepsilon}_R|^2 +C\varepsilon^4.
\end{split}
\end{equation*}
We note that
\begin{equation*}
\begin{split}
C\varepsilon^6\|\partial_xu^{\varepsilon}_R\|_{L^{\infty}}^2\|\partial_xn^{\varepsilon}_R\|^2\leq C\varepsilon^6\|u^{\varepsilon}_R\|_{H^2}^2\|n^{\varepsilon}_R\|_{H^1}^2\leq C(\varepsilon\tilde C)\|u^{\varepsilon}_R\|_{H^2}^2.
\end{split}
\end{equation*}
The Lemma then follows form Lemma \ref{L1}.
\end{proof}

\begin{lemma}\label{L3}
Let $(n^{\varepsilon}_R,u^{\varepsilon}_R,\phi^{\varepsilon}_R)$ be a solution to \eqref{rem2} and $\alpha\geq 0$ be an integer. There exist some constant $C_1=C(\varepsilon\tilde C)$ and $\varepsilon_1>0$ such that for any $0<\varepsilon<\varepsilon_1$,
\begin{equation*}
\begin{split}
\varepsilon\|\partial_t\partial_x^{\alpha+1}\phi^{\varepsilon}_R\|^2 +\|\partial_t\partial_x^{\alpha}\phi^{\varepsilon}_R\|^2 \leq  2\|\partial_t\partial_x^{\alpha}n^{\varepsilon}_R\|^2+C_1.
\end{split}
\end{equation*}
\end{lemma}
\begin{proof}
The proof is similar to that of Lemma \ref{L1}. When $\alpha=0$, by first taking $\partial_t$ of \eqref{rem2-3} and then taking inner product with $\partial_t\phi^{\varepsilon}_R$, we have
\begin{equation*}
\begin{split}
\varepsilon\|\partial_{tx}\phi^{\varepsilon}_R\|^2  +\|\partial_{t}\phi^{\varepsilon}_R\|^2 =&\int\partial_tn^{\varepsilon}_R\partial_t\phi^{\varepsilon}_R -\int(\varepsilon\partial_t(\phi^{(1)}\phi^{\varepsilon}_R) +\varepsilon^2\partial_t\mathcal R_3)\partial_t\phi^{\varepsilon}_R\\
\leq & \frac14\|\partial_{t}\phi^{\varepsilon}_R\|^2+\|\partial_{t}n^{\varepsilon}_R\|^2 +C({\varepsilon}\tilde C)\varepsilon(\|\phi^{\varepsilon}_R\|^2+\|\partial_t\phi^{\varepsilon}_R\|^2),
\end{split}
\end{equation*}
thanks to \eqref{equ34} in Lemma \ref{L8} in Appendix. Therefore, there exists some $\varepsilon_1>0$ such that when $\varepsilon<\varepsilon_1$,
\begin{equation*}
\begin{split}
\varepsilon\|\partial_{tx}\phi^{\varepsilon}_R\|^2  +\|\partial_{t}\phi^{\varepsilon}_R\|^2 \leq & 2\|\partial_{t}n^{\varepsilon}_R\|^2 +C({\varepsilon}\tilde C).
\end{split}
\end{equation*}

When $\alpha=1$, we take $\partial_{tx}$ of \eqref{rem2-3} and then take inner product with $\partial_{tx}\phi^{\varepsilon}_R$ to obtain
\begin{equation*}
\begin{split}
\varepsilon\|\partial_t\partial_x^{2}\phi^{\varepsilon}_R\|^2 +\|\partial_{tx}\phi^{\varepsilon}_R\|^2 \leq  2\|\partial_{tx}n^{\varepsilon}_R\|^2 +C({\varepsilon}\tilde C).
\end{split}
\end{equation*}
The case of $\alpha\geq 2$ can be proved similarly.
\end{proof}

The rest of this section is devoted to the proof of Theorem \ref{thm1} for the case $T_i=0$, which is divided into the following several subsections.

\subsection{Zeroth, first and second order estimates}
\begin{proposition}\label{L4}
Let $(n^{\varepsilon}_R,u^{\varepsilon}_R,\phi^{\varepsilon}_R)$ be a solution to \eqref{rem2} and $\gamma=0,1,2$, then
\begin{equation}\label{equ58}
\begin{split}
\frac12\frac{d}{dt}&\|\partial_x^{\gamma}u^{\varepsilon}_R\|^2 +\frac12\frac{d}{dt}[\int\frac{1+\varepsilon\phi^{(1)}}{n} |\partial_x^{\gamma}\phi^{\varepsilon}_R|^2 +\int\frac{\varepsilon}{n}|\partial_x^{\gamma+1}\phi^{\varepsilon}_R|^2] \\
&\leq C_1(1+\varepsilon^2|\!|\!|(u^{\varepsilon}_R,\phi^{\varepsilon}_R) |\!|\!|^2_{\varepsilon}) (1+|\!|\!|(u^{\varepsilon}_R,\phi^{\varepsilon}_R)|\!|\!|^2_{\varepsilon}).
\end{split}
\end{equation}
\end{proposition}
\begin{proof}
We take $\partial_x^{\gamma}$ of \eqref{rem2-2} and then take inner product of $\partial_x^{\gamma}u^{\varepsilon}_R$. Integrating by parts, we obtain
\begin{equation}\label{u1}
\begin{split}
\frac12\frac{d}{dt}\|\partial_{x}^{\gamma}u^{\varepsilon}_R\|^2 & -\frac{1}{\varepsilon}\int\partial_{x}^{\gamma+1}u^{\varepsilon}_R \partial_{x}^{\gamma}u^{\varepsilon}_R +\int\partial_{x}^{\gamma}\left[(\tilde u+\varepsilon^2u^{\varepsilon}_R)\partial_{x}u^{\varepsilon}_R\right] \partial_{x}^{\gamma}u^{\varepsilon}_R\\
&+\int\partial_{x}^{\gamma}\left[\partial_x\tilde uu^{\varepsilon}_R\right]\partial_{x}^{\gamma}u^{\varepsilon}_R +\int\partial_{x}^{\gamma}\left[\varepsilon\mathcal R_2\right]\partial_{x}^{\gamma}u^{\varepsilon}_R\\ =&\int\partial_x^{\gamma}\phi^{\varepsilon}_R \frac{\partial_{x}^{\gamma+1}u^{\varepsilon}_R}{\varepsilon}=:I^{(\gamma)}.
\end{split}
\end{equation}

\emph{Estimate of the LHS of \eqref{u1}.} The second term on the LHS of \eqref{u1} vanishes by integration by parts. The third term on the LHS of \eqref{u1} consists of two parts. For the first part, for $0\leq\gamma\leq 2$, we have
\begin{equation*}
\begin{split}
\int\partial_{x}^{\gamma}(\tilde u\partial_{x}u^{\varepsilon}_R) \partial_{x}^{\gamma}u^{\varepsilon}_R=&\int\tilde u\partial_{x}^{\gamma+1}u^{\varepsilon}_R\partial_{x}^{\gamma}u^{\varepsilon}_R +\sum_{0\leq\beta\leq \gamma-1}C_{\gamma}^{\beta}\int \partial_{x}^{\gamma-\beta}\tilde u\partial_{x}^{\beta+1}u^{\varepsilon}_R\partial_{x}^{\gamma}u^{\varepsilon}_R\\
=&-\frac12\int \partial_x\tilde u\partial_{x}^{\gamma}u^{\varepsilon}_R\partial_{x}^{\gamma}u^{\varepsilon}_R +\sum_{0\leq\beta\leq \gamma-1}C_{\gamma}^{\beta}\int \partial_{x}^{\gamma-\beta}\tilde u\partial_{x}^{\beta+1}u^{\varepsilon}_R\partial_{x}^{\gamma}u^{\varepsilon}_R\\
\leq & C\|u^{\varepsilon}_R\|^2_{H^2},
\end{split}
\end{equation*}
where, when $\gamma=0$, there is no such ``summation" term. For the second part, after integration by parts, we have for $0\leq \gamma\leq 2$
\begin{equation*}
\begin{split}
\varepsilon^2\int&\partial_{x}^{\gamma}(u^{\varepsilon}_R\partial_{x}u^{\varepsilon}_R) \partial_{x}^{\gamma}u^{\varepsilon}_R\\
=&-\frac{\varepsilon^2}{2}\int \partial_xu^{\varepsilon}_R\partial_{x}^{\gamma}u^{\varepsilon}_R \partial_{x}^{\gamma}u^{\varepsilon}_R +\sum_{0\leq\beta\leq \gamma-1}C_{\gamma}^{\beta}\varepsilon^2\int \partial_{x}^{\gamma-\beta}u^{\varepsilon}_R \partial_{x}^{\beta+1}u^{\varepsilon}_R\partial_{x}^{\gamma}u^{\varepsilon}_R\\
\leq & C\varepsilon^2\|\partial_xu^{\varepsilon}_R\|_{L^{\infty}} \|u^{\varepsilon}_R\|_{H^{\gamma}}^2\\
\leq & C(\varepsilon^2|\!|\!|(u^{\varepsilon}_R,\phi^{\varepsilon}_R) |\!|\!|_{\varepsilon})\|u^{\varepsilon}_R\|_{H^{\gamma}}^2,
\end{split}
\end{equation*}
where $|\!|\!|(u^{\varepsilon}_R,\phi^{\varepsilon}_R) |\!|\!|_{\varepsilon}$ is given in \eqref{def-A}. For the last two terms on the LHS of \eqref{u1}, since $\partial_x^{\gamma}\mathcal R_2$ is integrable by \eqref{rrr-3} and Theorem \ref{thm3}, they can be similarly bounded by $\|u^{\varepsilon}_R\|_{H^{\gamma}}^2+C\varepsilon^2.$
In summary, the last four terms on the LHS of \eqref{u1} are bounded by
\begin{equation}\label{left}
\begin{split}
C(1+\varepsilon^2|\!|\!|(u^{\varepsilon}_R,\phi^{\varepsilon}_R) |\!|\!|_{\varepsilon})(1+\|u^{\varepsilon}_R\|_{H^{\gamma}}^2).
\end{split}
\end{equation}

\emph{Estimate of the RHS term $I^{(\gamma)}$ in \eqref{u1}.} Taking $\partial_x^{\gamma}$ of \eqref{rem2-1}, we have
\begin{align}\label{equ59}
\frac{\partial_{x}^{\gamma+1}u^{\varepsilon}_R}{\varepsilon} =&\frac1n\bigg[\frac{(1-u)}{\varepsilon}\partial_{x}^{\gamma+1}n^{\varepsilon}_R -\partial_{t}\partial_x^{\gamma}n^{\varepsilon}_R -\sum_{0\leq\beta\leq\gamma-1}C_{\gamma}^{\beta}\partial_x^{\gamma-\beta}(\tilde n+\varepsilon^2 n^{\varepsilon}_R)\partial_x^{\beta+1}u^{\varepsilon}_R\nonumber\\ &-\sum_{0\leq\beta\leq\gamma-1}C_{\gamma}^{\beta}\partial_x^{\gamma-\beta}(\tilde u+\varepsilon^2 u^{\varepsilon}_R)\partial_x^{\beta+1}n^{\varepsilon}_R -\sum_{0\leq\beta\leq\gamma}C_{\gamma}^{\beta}\partial_x^{\beta}u^{\varepsilon}_R \partial_x^{\gamma-\beta+1}\tilde n\\ &-\sum_{0\leq\beta\leq\gamma}C_{\gamma}^{\beta}\partial_x^{\beta}n^{\varepsilon}_R \partial_x^{\gamma-\beta+1}\tilde u -\varepsilon\partial_x^{\gamma}\mathcal R_1 \bigg]=:\sum_{i=1}^7A_i^{(\gamma)}.\nonumber
\end{align}
Accordingly, $I^{(\gamma)}$ is decomposed into
\begin{equation}\label{equ60}
\begin{split}
I^{(\gamma)}=\sum_{i=1}^7I^{(\gamma)}_{i}=\sum_{i=1}^7\int\partial_x^2 \phi^{\varepsilon}_R A_i^{(\gamma)}.
\end{split}
\end{equation}
We first estimate the terms $I^{(\gamma)}_{i}$ for $3\leq i\leq 7$ and leave $I^{(\gamma)}_{1}$ and $I^{(\gamma)}_{2}$ in the following two lemmas.

\emph{Estimate of $I^{(\gamma)}_{3}$ in \eqref{equ60}.} We divide it into two parts
\begin{equation*}
\begin{split}
I^{(\gamma)}_{3}=&\sum_{0\leq\beta\leq\gamma-1}C_{\gamma}^{\beta} \int\partial_x^2\phi^{\varepsilon}_R\partial_x^{\gamma-\beta}\tilde n\partial_x^{\beta+1}u^{\varepsilon}_R +\sum_{0\leq\beta\leq\gamma-1}C_{\gamma}^{\beta}\varepsilon^2 \int\partial_x^2\phi^{\varepsilon}_R\partial_x^{\gamma-\beta}n^{\varepsilon}_R \partial_x^{\beta+1}u^{\varepsilon}_R\\
=&:I^{(\gamma)}_{31}+I^{(\gamma)}_{32}.
\end{split}
\end{equation*}
The first one is easily bounded by
\begin{equation*}
\begin{split}
I^{(\gamma)}_{31}\leq C(\|u^{\varepsilon}_R\|_{H^2}^2+\|\phi^{\varepsilon}_R\|_{H^2}^2).
\end{split}
\end{equation*}
For the second term $I_{32}^{\gamma}$, since the order of the derivative on $n^{\varepsilon}_R$ does not exceed 2, using H\"older inequality, Sobolev embedding $H^1\hookrightarrow L^{\infty}$ and then Lemma \ref{L1}, we deduce
\begin{equation*}
\begin{split}
I^{(\gamma)}_{32}\leq & C\varepsilon^2\|\partial_x^2\phi^{\varepsilon}_R\|_{L^{\infty}} \|n^{\varepsilon}_R\|_{H^2} \|u^{\varepsilon}_R\|_{H^2} \\
\leq & C_1(1+\varepsilon^2|\!|\!|(u^{\varepsilon}_R,\phi^{\varepsilon}_R)|\!|\!|_{\varepsilon}^2) (\|u^{\varepsilon}_R\|_{H^2}^2+\varepsilon\|\phi^{\varepsilon}_R\|_{H^3}^2),
\end{split}
\end{equation*}
where $|\!|\!|(u^{\varepsilon}_R,\phi^{\varepsilon}_R) |\!|\!|_{\varepsilon}$ is given in \eqref{def-A}. Hence
\begin{equation*}
\begin{split}
I^{(\gamma)}_{3} \leq C_1(1+\varepsilon^2|\!|\!|(u^{\varepsilon}_R,\phi^{\varepsilon}_R)|\!|\!|_{\varepsilon}^2) |\!|\!|(u^{\varepsilon}_R,\phi^{\varepsilon}_R)|\!|\!|_{\varepsilon}^2.
\end{split}
\end{equation*}

\emph{Estimate of $I^{(\gamma)}_{4}$ in \eqref{equ60}.} The term $I^{(\gamma)}_{4}$ is bounded similarly
\begin{equation*}
\begin{split}
I^{(\gamma)}_{4} \leq C_1(1+\varepsilon^2|\!|\!|(u^{\varepsilon}_R,\phi^{\varepsilon}_R)|\!|\!|_{\varepsilon}^2) |\!|\!|(u^{\varepsilon}_R,\phi^{\varepsilon}_R)|\!|\!|_{\varepsilon}^2.
\end{split}
\end{equation*}

\emph{Estimate of $I^{(\gamma)}_{5},I^{(\gamma)}_{6},I^{(\gamma)}_{7}$ in \eqref{equ60}.} Since the terms $I^{(\gamma)}_{i}$ for $i=5,6,7$ are bilinear or linear in the unknowns, they can be bounded by
\begin{equation*}
\begin{split}
I^{(\gamma)}_{5,6,7} \leq C_1(1+|\!|\!|(u^{\varepsilon}_R,\phi^{\varepsilon}_R)|\!|\!|_{\varepsilon}^2).
\end{split}
\end{equation*}

In summary, we have
\begin{equation*}
\begin{split}
\sum_{i=3}^7I^{(\gamma)}_{i} \leq C_1(1+\varepsilon^2|\!|\!|(u^{\varepsilon}_R,\phi^{\varepsilon}_R)|\!|\!|_{\varepsilon}^2) (1+|\!|\!|(u^{\varepsilon}_R,\phi^{\varepsilon}_R)|\!|\!|_{\varepsilon}^2).
\end{split}
\end{equation*}

We deduce Proposition \ref{L4} by the following Lemma \ref{Lem-u1} and \ref{Lem-u2}.
\end{proof}

\begin{lemma}[\textbf{\emph{Estimate of $I^{(\gamma)}_{1}$}}]\label{Lem-u1}
Let $(n^{\varepsilon}_R,u^{\varepsilon}_R,\phi^{\varepsilon}_R)$ be a solution to \eqref{rem2}, we have
\begin{equation*}
\begin{split}
I^{(\gamma)}_{1}\leq & C_1(1+\varepsilon^2|\!|\!|(u^{\varepsilon}_R,\phi^{\varepsilon}_R)|\!|\!|_{\varepsilon}^2) (1+|\!|\!|(u^{\varepsilon}_R,\phi^{\varepsilon}_R)|\!|\!|_{\varepsilon}^2),
\end{split}
\end{equation*}
where $I^{(\gamma)}_{1}$ is defined in \eqref{equ60} and $|\!|\!|(u^{\varepsilon}_R,\phi^{\varepsilon}_R)|\!|\!|_{\varepsilon}$ is given in \eqref{def-A}.
\end{lemma}
\begin{proof}[Proof of Lemma \ref{Lem-u1}.]
Taking $\partial_x^{\gamma+1}$ of \eqref{rem2-3}, we have
\begin{equation*}
\begin{split}
\partial_{x}^{\gamma+1}n^{\varepsilon}_R=\partial_{x}^{\gamma+1}\phi^{\varepsilon}_R -\varepsilon\partial_x^{\gamma+3}\phi^{\varepsilon}_R +\varepsilon\partial_{x}^{\gamma+1}(\phi^{(1)}\phi^{\varepsilon}_R) +\varepsilon^2\partial_{x}^{\gamma+1}{\mathcal R}_3=:\sum_{i=1}^4B^{(\gamma)}_i.
\end{split}
\end{equation*}
Accordingly, $I^{(\gamma)}_1$ is decomposed into
\begin{equation*}
\begin{split}
I^{(\gamma)}_1=\sum_{i=1}^4\int\partial_x^{\gamma}\phi^{\varepsilon}_R\left[\frac{(1-u)}{\varepsilon n}B^{(\gamma)}_i\right]=:\sum_{i=1}^4I^{(\gamma)}_{1i}.
\end{split}
\end{equation*}

\emph{Estimate of $I^{(\gamma)}_{11}$.} Integrating by parts yields
\begin{equation*}
\begin{split}
I^{(\gamma)}_{11}=&\int\frac{(1-u)}{\varepsilon n}\partial_x^{\gamma}\phi^{\varepsilon}_R\partial_x^{\gamma+1}\phi^{\varepsilon}_R\\ =&-\frac12\int\partial_x\left[\frac{(1-u)}{\varepsilon n}\right]|\partial_x^{\gamma}\phi^{\varepsilon}_R|^2\\
\leq & C\|\partial_x^{\gamma}\phi^{\varepsilon}_R\|^2 +C\varepsilon^2(\|\partial_xu^{\varepsilon}_R\|_{L^{\infty}} +\|\partial_xn^{\varepsilon}_R\|_{L^{\infty}})\|\partial_x^{\gamma}\phi^{\varepsilon}_R\|^2,
\end{split}
\end{equation*}
thanks to the fact
\begin{equation*}
\begin{split}
\partial_x\left[\frac{(1-u)}{\varepsilon n}\right]\leq C((|\partial_x\tilde n|+|\partial_x\tilde u|) +\varepsilon^2(|\partial_xn^{\varepsilon}_R|+|\partial_xu^{\varepsilon}_R|)).
\end{split}
\end{equation*}
Using Sobolev embedding and Lemma \ref{L1}, by \eqref{def-A} we have
\begin{equation}\label{equ61}
\begin{split}
I^{(\gamma)}_{11} \leq & C\|\partial_x^{\gamma}\phi^{\varepsilon}_R\|^2 +C_1(1+\varepsilon|\!|\!|(u^{\varepsilon}_R, u^{\varepsilon}_R)|\!|\!|_{\varepsilon})\|\partial_x^{\gamma}\phi^{\varepsilon}_R\|^2.
\end{split}
\end{equation}

\emph{Estimate of $I^{(\gamma)}_{12}$.} By integration by parts twice, we have
\begin{equation}\label{equ62}
\begin{split}
I^{(\gamma)}_{12}=&-\int\partial_x^{\gamma}\phi^{\varepsilon}_R\left[\frac{(1-u)}{n} \partial_x^{\gamma+3}\phi^{\varepsilon}_R\right]\\
=&-\frac32\int\partial_x\left[\frac{(1-u)}{n}\right] |\partial_x^{\gamma+1}\phi^{\varepsilon}_R|^2 -\int\partial_x^2\left[\frac{(1-u)}{n}\right] \partial_x^{\gamma}\phi^{\varepsilon}_R \partial_x^{\gamma+1}\phi^{\varepsilon}_R\\
=&:I^{(\gamma)}_{121}+I^{(\gamma)}_{122}.
\end{split}
\end{equation}
Note that
\begin{equation*}
\begin{split}
\partial_x\left[\frac{(V-u)}{n}\right]\leq C\Big(\varepsilon(|\partial_x\tilde n|+|\partial_x\tilde u|) +\varepsilon^3(|\partial_xn^{\varepsilon}_R|+|\partial_xu^{\varepsilon}_R|)\Big).
\end{split}
\end{equation*}
Similar to the bound for $I^{(\gamma)}_{11}$ in \eqref{equ61}, we have
\begin{equation}\label{equ63}
\begin{split}
I^{(\gamma)}_{121} \leq & C\varepsilon\|\partial_x^{\gamma+1}\phi^{\varepsilon}_R\|^2 +C_1(1+\varepsilon|\!|\!|(u^{\varepsilon}_R, u^{\varepsilon}_R)|\!|\!|_{\varepsilon}) (\varepsilon\|\partial_x^{\gamma+1}\phi^{\varepsilon}_R\|^2).
\end{split}
\end{equation}
Note that
\begin{equation*}
\begin{split}
\left|\partial_x^2\left[\frac{(1-u)}{n}\right]\right|\leq & C\Big(\varepsilon +\varepsilon^3(|\partial_x^2n^{\varepsilon}_R|+|\partial_x^2u^{\varepsilon}_R|)\\
& +\varepsilon^4(|\partial_xn^{\varepsilon}_R|+|\partial_xu^{\varepsilon}_R|) +\varepsilon^6(|\partial_xn^{\varepsilon}_R|^2+|\partial_xu^{\varepsilon}_R|^2)\Big).
\end{split}
\end{equation*}
By H\"older inequality, Sobolev embedding and Lemma \ref{L1} for $0\leq\gamma\leq2$, we obtain
\begin{equation}\label{equ64}
\begin{split}
I^{(\gamma)}_{122}\leq & C\|\partial_x^{\gamma}\phi^{\varepsilon}_R\|_{L^{\infty}}\|\partial_x^2\left[\frac{(V-u)}{n}\right]\| \|\partial_x^{\gamma+1}\phi^{\varepsilon}_R\|\\
\leq & C\varepsilon\|\partial_x^{\gamma}\phi^{\varepsilon}_R\|_{H^1} (1+\varepsilon^2(\|n^{\varepsilon}_R\|_{H^2}^2+\|u^{\varepsilon}_R\|_{H^2}^2)) \|\partial_x^{\gamma+1}\phi^{\varepsilon}_R\|\\
\leq & C_1(1+\varepsilon^2|\!|\!|(u^{\varepsilon}_R, u^{\varepsilon}_R)|\!|\!|_{\varepsilon}^2) (\varepsilon\|\phi^{\varepsilon}_R\|^2_{H^{\gamma+1}}).
\end{split}
\end{equation}
Therefore, combining \eqref{equ62}, \eqref{equ63} and \eqref{equ64}, we obtain
\begin{equation*}
\begin{split}
I^{(\gamma)}_{12}\leq C_1(1+\varepsilon^2|\!|\!|(u^{\varepsilon}_R, u^{\varepsilon}_R)|\!|\!|_{\varepsilon}^2) (\varepsilon\|\phi^{\varepsilon}_R\|^2_{H^{\gamma+1}}).
\end{split}
\end{equation*}

\emph{Estimate for $I^{(\gamma)}_{13}$.} The estimate for $I^{(\gamma)}_{13}$ is similar to that for $I^{(\gamma)}_{11}$ in \eqref{equ61},
\begin{equation*}
\begin{split}
I^{(\gamma)}_{13} \leq & C_1(1+\varepsilon|\!|\!|(u^{\varepsilon}_R, u^{\varepsilon}_R)|\!|\!|_{\varepsilon})\|\phi^{\varepsilon}_R\|_{H^{\gamma}}^2.
\end{split}
\end{equation*}

\emph{Estimate of $I^{(\gamma)}_{14}$.} By integration by parts and Lemma \ref{L8}, we deduce
\begin{equation*}
\begin{split}
I^{(2)}_{14}=&\varepsilon\int\partial_x^{\gamma}\phi^{\varepsilon}_R \left[\frac{(1-u)}{n} \partial_{x}^{\gamma+1}{\mathcal R}_3\right]\\
=&-\varepsilon\int\partial_x\left[\frac{(1-u)}{n}\right] \partial_x^{\gamma}\phi^{\varepsilon}_R\partial_x^{\gamma}\mathcal R_3-\varepsilon\int\frac{(1-u)}{n}\partial_x^{\gamma+1}\phi^{\varepsilon}_R \partial_x^{\gamma}\mathcal R_3\\
\leq & C_1\varepsilon^2(1+\|\partial_xn^{\varepsilon}_R\|_{L^{\infty}} +\|\partial_xu^{\varepsilon}_R\|_{L^{\infty}})\|\phi^{\varepsilon}_R\|_{H^{\gamma}}^2 +C_1\varepsilon\|\phi^{\varepsilon}_R\|_{H^{\gamma+1}}\|\phi^{\varepsilon}_R\|_{H^{\gamma}}\\
\leq & C_1(1+\varepsilon|\!|\!|(u^{\varepsilon}_R, u^{\varepsilon}_R)|\!|\!|_{\varepsilon})(\varepsilon\|\phi^{\varepsilon}_R\|_{H^{\gamma+1}}^2).
\end{split}
\end{equation*}
The proof of Lemma \ref{Lem-u1} is then complete.
\end{proof}

\begin{lemma}[\textbf{\emph{Estimate of $I^{(\gamma)}_{2}$}}]\label{Lem-u2} Let $(n^{\varepsilon}_R,u^{\varepsilon}_R,\phi^{\varepsilon}_R)$ be a solution to \eqref{rem2} and $0\leq\gamma\leq2$. The following inequality holds
\begin{equation*}
\begin{split}
I_{2}^{(\gamma)}\leq & -\frac12\frac{d}{dt}\int\frac{1+\varepsilon\phi^{(1)}}{n} |\partial_x^{\gamma}\phi^{\varepsilon}_R|^2dx -\frac12\frac{d}{dt}\int\frac{\varepsilon}n |\partial_x^{\gamma+1}\phi^{\varepsilon}_R|^2dx\\
& +C_1(1 +\varepsilon^2|\!|\!|(u^{\varepsilon}_R,\phi^{\varepsilon}_R)|\!|\!|^2_{\varepsilon}) (1+|\!|\!|(u^{\varepsilon}_R,\phi^{\varepsilon}_R)|\!|\!|^2_{\varepsilon}),
\end{split}
\end{equation*}
where $I^{(\gamma)}_{2}$ is defined in \eqref{equ60} and $|\!|\!|(u^{\varepsilon}_R,\phi^{\varepsilon}_R)|\!|\!|_{\varepsilon}$ is given in \eqref{def-A}.
\end{lemma}
\begin{proof}[Proof of Lemma \ref{Lem-u2}.]
Taking $\partial_x\partial_x^{\gamma}$ of \eqref{rem2-3}, we have
\begin{equation*}
\begin{split}
\partial_t\partial_{x}^{\gamma}n^{\varepsilon}_R= \partial_t\partial_{x}^{\gamma}\phi^{\varepsilon}_R -\varepsilon\partial_t\partial_x^{\gamma+2}\phi^{\varepsilon}_R +\varepsilon\partial_t\partial_{x}^{\gamma}(\phi^{(1)}\phi^{\varepsilon}_R) +\varepsilon^2\partial_t\partial_{x}^{\gamma}{\mathcal R}_3=:\sum_iD^{(2)}_i.
\end{split}
\end{equation*}
Accordingly, we have the decomposition
\begin{equation}\label{equ65}
\begin{split}
I^{(\gamma)}_2=&-\sum_{i=1}^4\int\frac{1}{n}\partial_x^{\gamma}\phi^{\varepsilon}_R D_i=:\sum_{i=1}^4I^{(\gamma)}_{2i}.
\end{split}
\end{equation}

\emph{Estimate of $I^{({\gamma})}_{21}$.} By integration by parts, we obtain
\begin{equation*}
\begin{split}
I_{21}^{({\gamma})}=&-\int\frac1n\partial_{x}^{\gamma}\phi^{\varepsilon}_R \partial_{t}\partial_x^{\gamma}\phi^{\varepsilon}_R\\
=&-\frac12\frac{d}{dt}\int\frac1n|\partial_{x}^{\gamma}\phi^{\varepsilon}_R|^2 +\frac12\int\partial_t\left[\frac1n\right]|\partial_{x}^{\gamma}\phi^{\varepsilon}_R|^2,
\end{split}
\end{equation*}
where the second term on the RHS is bounded by Lemma \ref{L2}
\begin{equation*}
\begin{split}
\frac12\int\partial_t\left[\frac1n\right]|\partial_{x}^{\gamma}\phi^{\varepsilon}_R|^2=& -\frac12\int\left[\frac{\varepsilon\partial_t\tilde n+\varepsilon^3\partial_tn^{\varepsilon}_R}{n^2}\right] |\partial_{x}^{\gamma}\phi^{\varepsilon}_R|^2\\
\leq & C\varepsilon\|\partial_{x}^{\gamma}\phi^{\varepsilon}_R\|^2 +\varepsilon^3\|\varepsilon\partial_tn^{\varepsilon}_R\|^2 \|\partial_{x}^{\gamma}\phi^{\varepsilon}_R\|^2 +\varepsilon\|\partial_{x}^{\gamma}\phi^{\varepsilon}_R\|^2_{L^{\infty}}\\
\leq & C_1(1+\varepsilon^2|\!|\!|(u^{\varepsilon}_R,\phi^{\varepsilon}_R) |\!|\!|^2_{\varepsilon})(\varepsilon\|\phi^{\varepsilon}_R\|^2_{H^{\gamma+1}}).
\end{split}
\end{equation*}
Hence
\begin{equation}\label{equ66}
\begin{split}
I_{21}^{(\gamma)}\leq &-\frac12\frac{d}{dt}\int\frac1n|\partial_{x}^{\gamma}\phi^{\varepsilon}_R|^2 +C_1(1+\varepsilon^2|\!|\!|(u^{\varepsilon}_R,\phi^{\varepsilon}_R) |\!|\!|^2_{\varepsilon})(\varepsilon\|\phi^{\varepsilon}_R\|^2_{H^{\gamma+1}}).
\end{split}
\end{equation}

\emph{Estimate of $I^{(\gamma)}_{22}$.} By integration by parts, we have
\begin{equation*}
\begin{split}
I_{22}^{(\gamma)}=&\int\frac{\varepsilon}{n}\partial_{x}^{\gamma}\phi^{\varepsilon}_R \partial_{t}\partial_x^{\gamma+2}\phi^{\varepsilon}_R\\
=&-\int\frac{\varepsilon}{n}\partial_{x}^{\gamma+1}\phi^{\varepsilon}_R \partial_{t}\partial_x^{\gamma+1}\phi^{\varepsilon}_R -\int\partial_x\left[\frac{\varepsilon}n\right]\partial_{x}^{\gamma}\phi^{\varepsilon}_R \partial_{t} \partial_x^{\gamma+1}\phi^{\varepsilon}_R\\
=& :I_{221}^{(\gamma)}+I_{222}^{(\gamma)}.
\end{split}
\end{equation*}
The first term is estimated by Sobolev embedding, Lemma \ref{L3} and \ref{L2} as
\begin{equation}\label{equ71}
\begin{split}
I_{221}^{({\gamma})}=&-\frac12\frac{d}{dt}\int\frac{\varepsilon}n |\partial_x^{\gamma+1}\phi^{\varepsilon}_R|^2 +\frac12\int\partial_t\left[\frac{\varepsilon}n\right] |\partial_x^{\gamma+1}\phi^{\varepsilon}_R|^2\\
\leq & -\frac12\frac{d}{dt}\int\frac{\varepsilon}n |\partial_x^{\gamma+1}\phi^{\varepsilon}_R|^2+ C\varepsilon(1+\varepsilon^2\|\partial_tn^{\varepsilon}_R\|_{L^{\infty}}) (\varepsilon\|\partial_x^{\gamma+1}\phi^{\varepsilon}_R\|^2)\\
\leq & -\frac12\frac{d}{dt}\int\frac{\varepsilon}n |\partial_x^{\gamma+1}\phi^{\varepsilon}_R|^2+ C_1\varepsilon(1+\varepsilon|\!|\!|(u^{\varepsilon}_R,\phi^{\varepsilon}_R) |\!|\!|_{\varepsilon}) (\varepsilon\|\partial_x^{\gamma+1}\phi^{\varepsilon}_R\|^2),
\end{split}
\end{equation}
where $|\!|\!|(u^{\varepsilon}_R,\phi^{\varepsilon}_R)|\!|\!|_{\varepsilon}$ is defined in \eqref{def-A}.

For $I_{222}^{(\gamma)}$, integration by parts yields
\begin{equation}\label{equ-B2}
\begin{split}
I_{222}^{(\gamma)}=\underbrace{\int\partial_{x}^{\gamma+1}\phi^{\varepsilon}_R \partial_x\left[\frac{\varepsilon}n\right]\partial_{t} \partial_x^{\gamma}\phi^{\varepsilon}_R}_{:=\mathcal B^{(\gamma)}} +\underbrace{\int\partial_{x}^{\gamma}\phi^{\varepsilon}_R \partial_x^{2}\left[\frac{\varepsilon}n\right] \partial_{t}\partial_x^{\gamma}\phi^{\varepsilon}_R}_{I^{(\gamma)}_{2221}}.
\end{split}
\end{equation}
We first bound $\mathcal B^{(\gamma)}$ in \eqref{equ-B2}. Since $0\leq \gamma\leq 2$, by Lemma \ref{L3} with $\alpha=1$, and Lemma \ref{L2} and \ref{L1}, we have
\begin{equation}\label{equ38}
\begin{split}
\mathcal B^{(\gamma)}=&-\int\varepsilon(\frac{\partial_x\tilde n+\varepsilon^2\partial_xn^{\varepsilon}_R}{n^2})\partial_x^{\gamma+1}\phi^{\varepsilon}_R (\varepsilon\partial_t\partial_x^{\gamma}\phi^{\varepsilon}_R)dx \\
\leq & C\varepsilon\|\varepsilon\partial_{t}\partial_x^{\gamma}\phi^{\varepsilon}_R\|^2 +C(1+\varepsilon^4\|\partial_xn^{\varepsilon}_R\|^2_{L^{\infty}}) (\varepsilon\|\partial_x^{\gamma+1}\phi^{\varepsilon}_R\|^2) \\
\leq & C_1(1+|\!|\!|(u^{\varepsilon}_R,\phi^{\varepsilon}_R)|\!|\!|^2_{\varepsilon}) +C_1(1+\varepsilon^2|\!|\!|(u^{\varepsilon}_R,\phi^{\varepsilon}_R) |\!|\!|^2_{\varepsilon}) (\varepsilon\|\phi^{\varepsilon}_R\|_{H^{\gamma+1}}^2).
\end{split}
\end{equation}
We now estimate $I^{(\gamma)}_{2221}$ in \eqref{equ-B2}. Note that
\begin{equation*}
\begin{split}
\left|\partial_x^2\left[\frac{\varepsilon}{n}\right]\right|\leq & C\varepsilon^2(1+\varepsilon^2|\partial_x^2n^{\varepsilon}_R| +\varepsilon^3|\partial_xn^{\varepsilon}_R| +\varepsilon^5|\partial_xn^{\varepsilon}_R|^2).
\end{split}
\end{equation*}
By H\"older inequality and Lemma \ref{L3} with $\alpha=1$,
\begin{equation}\label{equ67}
\begin{split}
\varepsilon^2\int|\partial_{x}^{\gamma}\phi^{\varepsilon}_R| |\partial_{t}\partial_x^{\gamma}\phi^{\varepsilon}_R|
\leq & C\varepsilon\|\partial_{x}^{\gamma}\phi^{\varepsilon}_R\|^2 +C\varepsilon\|\varepsilon\partial_{t}\partial_x^{\gamma}\phi^{\varepsilon}_R\|^2\\
\leq & C\varepsilon\|\partial_{x}^{\gamma}\phi^{\varepsilon}_R\|^2 +C_1(1+|\!|\!|(u^{\varepsilon}_R,\phi^{\varepsilon}_R)|\!|\!|^2_{\varepsilon}).
\end{split}
\end{equation}
By Lemma \ref{L1}, Lemma \ref{L3} and \ref{L2},
\begin{equation}\label{equ68}
\begin{split}
\varepsilon^4\int&|\partial_{x}^{\gamma}\phi^{\varepsilon}_R| |\partial_{x}^2n^{\varepsilon}_R||\partial_{t}\partial_x^{\gamma}\phi^{\varepsilon}_R|\\
\leq & C\varepsilon^2\|\partial_x^2n^{\varepsilon}_R\|^2 (\varepsilon\|\partial_{x}^{\gamma}\phi^{\varepsilon}_R\|^2_{L^{\infty}}) +C\varepsilon^2 (\varepsilon\|\varepsilon\partial_{t}\partial_x^{\gamma}\phi^{\varepsilon}_R\|^2)\\
\leq & C_1\varepsilon^2|\!|\!|(u^{\varepsilon}_R,\phi^{\varepsilon}_R)|\!|\!|^2_{\varepsilon} (\varepsilon\|\phi^{\varepsilon}_R\|^2_{H^{\gamma+1}}) +C_1\varepsilon^2(1 +|\!|\!|(u^{\varepsilon}_R,\phi^{\varepsilon}_R)|\!|\!|^2_{\varepsilon}).
\end{split}
\end{equation}
By H\"older inequality and Lemma \ref{L3}, \ref{L2} and \ref{L1}
\begin{equation}\label{equ69}
\begin{split}
\varepsilon^5\int&|\partial_{x}^{\gamma}\phi^{\varepsilon}_R| (|\partial_{x}n^{\varepsilon}_R|+\varepsilon^2|\partial_{x}n^{\varepsilon}_R|^2) |\partial_{t}\partial_x^{\gamma}\phi^{\varepsilon}_R|\\
\leq & \varepsilon^2(1+\|\partial_{x}n^{\varepsilon}_R\|_{L^{\infty}}^2) (\varepsilon\|\varepsilon\partial_t\partial_{x}^{\gamma}\phi^{\varepsilon}_R\|^2) +C\varepsilon^2\|\partial_{x}n^{\varepsilon}_R\|_{L^{\infty}}^2 \|\partial_{x}^{\gamma}\phi^{\varepsilon}_R\|^2\\
\leq & C_1(1+\varepsilon^2|\!|\!|(u^{\varepsilon}_R,\phi^{\varepsilon}_R)|\!|\!|^2_{\varepsilon}) (1+|\!|\!|(u^{\varepsilon}_R,\phi^{\varepsilon}_R)|\!|\!|^2_{\varepsilon}).
\end{split}
\end{equation}
Summarizing \eqref{equ-B2}, \eqref{equ67}-\eqref{equ69} , we have
\begin{equation}\label{equ70}
\begin{split}
I^{(\gamma)}_{2221} \leq & C_1(1+\varepsilon^2|\!|\!|(u^{\varepsilon}_R,\phi^{\varepsilon}_R)|\!|\!|^2_{\varepsilon}) (1+|\!|\!|(u^{\varepsilon}_R,\phi^{\varepsilon}_R)|\!|\!|^2_{\varepsilon}).
\end{split}
\end{equation}
Therefore, by \eqref{equ71}, \eqref{equ38} and \eqref{equ70}, we obtain
\begin{equation*}
\begin{split}
I^{(\gamma)}_{22} \leq & -\frac12\frac{d}{dt}\int\frac{\varepsilon}n |\partial_x^{\gamma+1}\phi^{\varepsilon}_R|^2+C_1(1+\varepsilon^2|\!|\!|(u^{\varepsilon}_R,\phi^{\varepsilon}_R)|\!|\!|^2_{\varepsilon}) (1+|\!|\!|(u^{\varepsilon}_R,\phi^{\varepsilon}_R)|\!|\!|^2_{\varepsilon}).
\end{split}
\end{equation*}

\emph{Estimate of $I^{(\gamma)}_{23}$.} The estimate for $I^{(\gamma)}_{23}$ in \eqref{equ65} is no more difficult than that of $I^{(\gamma)}_{21}$ and can be bounded by
\begin{equation*}
\begin{split}
I_{23}^{(\gamma)}\leq &-\frac12\frac{d}{dt}\int\frac{\varepsilon\phi^{(1)}}n |\partial_{x}^{\gamma}\phi^{\varepsilon}_R|^2 +C(1+\varepsilon|\!|\!|(u^{\varepsilon}_R,\phi^{\varepsilon}_R) |\!|\!|^2_{\varepsilon})(\varepsilon\|\phi^{\varepsilon}_R\|^2_{H^3}).
\end{split}
\end{equation*}
\emph{Estimate of $I^{(\gamma)}_{24}$.} By using Lemma \ref{L8}, and then Lemma \ref{L3} with $\alpha=1$ and Lemma \ref{L2}, we have
\begin{equation*}
\begin{split}
I^{(\gamma)}_{24}=&-\int\frac{\varepsilon^2}{n}\partial_x^{\gamma}\phi^{\varepsilon}_R \partial_t\partial_{x}^{\gamma}{\mathcal R}_3\\
\leq & C\|\partial_x^{\gamma}\phi^{\varepsilon}_R\|^2+\varepsilon C(\|\phi^{(i)}\|_{H^{\tilde s_i}},\varepsilon\|\phi^{\varepsilon}_R\|_{H^{2}}) (\varepsilon\|\varepsilon\partial_t\phi^{\varepsilon}_R\|_{H^{\gamma}}^2)\\
\leq & C_1(1+|\!|\!|(u^{\varepsilon}_R,\phi^{\varepsilon}_R) |\!|\!|^2_{\varepsilon}).
\end{split}
\end{equation*}
The proof of Lemma \ref{Lem-u2} is then complete.
\end{proof}

When $\gamma=2$, by extracting the term $\mathcal B^{(2)}$ from \eqref{equ-B2}, we have the following
\begin{corollary}\label{Cor}
Let $(n^{\varepsilon}_R,u^{\varepsilon}_R,\phi^{\varepsilon}_R)$ be a solution to \eqref{rem2} and $0\leq\gamma\leq2$, then
\begin{equation*}
\begin{split}
\frac12\frac{d}{dt}&\bigg[\|\partial_x^2u^{\varepsilon}_R\|^2\bigg] +\frac12\frac{d}{dt}\bigg[\bigg(\int\frac{1+\varepsilon\phi^{(1)}}{n} |\partial_x^2\phi^{\varepsilon}_R|^2 +\int\frac{\varepsilon}{n}|\partial_x^3\phi^{\varepsilon}_R|^2\bigg) \bigg] \\
&\leq   C_1(1+\varepsilon^2|\!|\!|(u^{\varepsilon}_R,\phi^{\varepsilon}_R) |\!|\!|^2_{\varepsilon}) (1+|\!|\!|(u^{\varepsilon}_R,\phi^{\varepsilon}_R)|\!|\!|^2_{\varepsilon}) +\mathcal B^{(2)},
\end{split}
\end{equation*}
where
\begin{equation*}
\begin{split}
\mathcal B^{(2)}=\int\partial_x\left[\frac{\varepsilon}n\right]\partial_{x}^3\phi^{\varepsilon}_R \partial_{t} \partial_x^2\phi^{\varepsilon}_R.
\end{split}
\end{equation*}
\end{corollary}
\begin{proof}
This follows from \eqref{equ58} with $\gamma=2$.
\end{proof}
We remark that the precise form of $\mathcal B^{(2)}$ is very important for us to close the proof later. Indeed, when $\gamma=3$, the term $\mathcal B^{(3)}$ is not controllable in terms of $|\!|\!|(u^{\varepsilon}_R,\phi^{\varepsilon}_R) |\!|\!|_{\varepsilon}$. We need an exact cancellation by $\mathcal B^{(2)}+\varepsilon\mathcal B^{(3)}$ (see Remark \ref{remark} below). This is the reason why the third order derivatives are estimated separately.

\subsection{Third order estimates}
\begin{proposition}\label{L7}
Let $(n^{\varepsilon}_R,u^{\varepsilon}_R,\phi^{\varepsilon}_R)$ be a solution to \eqref{rem2}, then
\begin{equation*}
\begin{split}
\frac12\frac{d}{dt}\bigg[\varepsilon\|\partial_x^3u^{\varepsilon}_R\|^2\bigg] +&\frac12\frac{d}{dt}\bigg[\bigg(\int\frac{\varepsilon(1+\varepsilon\phi^{(1)})}{n} |\partial_x^3\phi^{\varepsilon}_R|^2 +\int\frac{\varepsilon^2}{n}|\partial_x^4\phi^{\varepsilon}_R|^2\bigg) \bigg] \\
\leq & C_1(1+\varepsilon^2|\!|\!|(u^{\varepsilon}_R,\phi^{\varepsilon}_R) |\!|\!|^2_{\varepsilon}) (1+|\!|\!|(u^{\varepsilon}_R,\phi^{\varepsilon}_R)|\!|\!|^2_{\varepsilon})+\mathcal B^{(3\times\varepsilon)},
\end{split}
\end{equation*}
where
\begin{equation*}
\begin{split}
\mathcal B^{(3\times\varepsilon)}=-\int\partial_x^3\phi^{\varepsilon}_R\partial_x\left[\frac{\varepsilon^2}{n}\right] \partial_{t}\partial_x^4\phi^{\varepsilon}_R.
\end{split}
\end{equation*}
\end{proposition}
\begin{proof}
Taking $\partial_x^3$ of \eqref{rem2-2} and then taking inner product with $\varepsilon\partial_x^3u^{\varepsilon}_R$, we have
\begin{equation}\label{3rd}
\begin{split}
\frac12\frac{d}{dt}(\varepsilon\|\partial_{x}^3u^{\varepsilon}_R\|^2)  -&\int\partial_{x}^4u^{\varepsilon}_R\partial_{x}^3u^{\varepsilon}_R +\int\varepsilon\partial_{x}^3\left[(\tilde u+\varepsilon u^{\varepsilon}_R)\partial_{x}u^{\varepsilon}_R\right]\partial_{x}^3u^{\varepsilon}_R\\ &+\int\varepsilon\partial_{x}^3\left[\partial_x\tilde uu^{\varepsilon}_R\right]\partial_{x}^3u^{\varepsilon}_R +\int\partial_{x}^3\left[\varepsilon^2\mathcal R_2\right]\partial_{x}^3u^{\varepsilon}_R \\ =&\int\varepsilon\partial_x^3\phi^{\varepsilon}_R \frac{\partial_{x}^4u^{\varepsilon}_R}{\varepsilon}
=:I^{(3\times\varepsilon)}.
\end{split}
\end{equation}

\emph{Estimate of LHS of \eqref{3rd}.} The second term on the LHS of \eqref{3rd} vanishes by integration by parts. For the third term, by expanding the derivatives and then integration by parts, we have
\begin{equation}\label{equ72}
\begin{split}
&\int\varepsilon\partial_{x}^3\left[(\tilde u+\varepsilon^2 u^{\varepsilon}_R)\partial_{x}u^{\varepsilon}_R\right]\partial_{x}^3u^{\varepsilon}_R\\ =&\frac52\int\varepsilon\partial_x(\tilde u+\varepsilon^2 u^{\varepsilon}_R)|\partial_{x}^3u^{\varepsilon}_R|^2 +\sum_{\beta=2,3}C_{3}^{\beta}\int\varepsilon\partial_{x}^\beta\left[(\tilde u+\varepsilon^2u^{\varepsilon}_R)\right]\partial_{x}^{4-\beta}u^{\varepsilon}_R \partial_{x}^3u^{\varepsilon}_R\\
=&:L_{31}^{(3\times\varepsilon)}+L_{32}^{(3\times\varepsilon)}.
\end{split}
\end{equation}
The first term $L_{31}^{(3\times\varepsilon)}$ on the RHS of \eqref{equ72} is estimated as
\begin{equation}\label{equ73}
\begin{split}
L_{31}^{(3\times\varepsilon)} \leq & C(1+\varepsilon\|\partial_xu^{\varepsilon}_R\|_{L^{\infty}}) (\varepsilon\|\partial_{x}^3u^{\varepsilon}_R\|^2)\\
\leq & C(1+\varepsilon|\!|\!|(u^{\varepsilon}_R,\phi^{\varepsilon}_R) |\!|\!|_{\varepsilon}) (\varepsilon\|\partial_{x}^3u^{\varepsilon}_R\|^2).
\end{split}
\end{equation}
When $\beta=3$, the second term $L_{32}^{(3\times\varepsilon)}$ on the RHS of \eqref{equ72} is estimated similarly to \eqref{equ73}. When $\beta=2$, the second term $L_{32}^{(3\times\varepsilon)}$ is estimated
\begin{equation*}
\begin{split}
L_{32}^{(3\times\varepsilon)}=&\int\varepsilon(\partial_{x}^2\tilde u+\varepsilon^2\partial_{x}^2u^{\varepsilon}_R)\partial_{x}^2u^{\varepsilon}_R \partial_{x}^3u^{\varepsilon}_R\\
\leq & C(1+\varepsilon^2\|\partial_x^2u^{\varepsilon}_R\|_{L^{\infty}}) (\varepsilon\|\partial_x^2u^{\varepsilon}_R\|\|\partial_{x}^3u^{\varepsilon}_R\|)\\
\leq & C(1+\varepsilon|\!|\!|(u^{\varepsilon}_R,\phi^{\varepsilon}_R) |\!|\!|_{\varepsilon})(\varepsilon\|u^{\varepsilon}_R\|_{H^3}^2).
\end{split}
\end{equation*}
By Lemma \ref{L8}, the last two terms on the LHS of \eqref{3rd} are easily bounded by
\begin{equation*}
\begin{split}
\varepsilon(1+\|u^{\varepsilon}_R\|_{H^3}^2).
\end{split}
\end{equation*}
Hence, the last four terms on the LHS of \eqref{3rd} are bounded by
\begin{equation*}
\begin{split}
C(1+\varepsilon|\!|\!|(u^{\varepsilon}_R,\phi^{\varepsilon}_R) |\!|\!|_{\varepsilon})(1+\varepsilon\|u^{\varepsilon}_R\|_{H^3}^2).
\end{split}
\end{equation*}

\emph{Decomposition of $I^{(3\times\varepsilon)}$ in \eqref{3rd}.} Taking $\partial_x^{3}$ of \eqref{rem2-1}, we have
\begin{equation}\label{equ74}
\begin{split}
\frac{\partial_{x}^4u^{\varepsilon}_R}{\varepsilon} =&\frac1n\bigg[\frac{(1-u)}{\varepsilon}\partial_{x}^4n^{\varepsilon}_R -\partial_{t}\partial_x^3n^{\varepsilon}_R -\sum_{\beta=1}^{3}C_{3}^{\beta}\partial_x^{\beta}(\tilde n+\varepsilon^2 n^{\varepsilon}_R)\partial_x^{4-\beta}u^{\varepsilon}_R\\ &-\sum_{\beta=1}^{3}C_{3}^{\beta}\partial_x^{\beta}(\tilde u+\varepsilon^2 u^{\varepsilon}_R)\partial_x^{4-\beta}n^{\varepsilon}_R -\sum_{\beta=0}^{3}C_{3}^{\beta}\partial_x^{\beta}u^{\varepsilon}_R \partial_x^{4-\beta}\tilde n\\ &-\sum_{\beta=0}^{3}C_{3}^{\beta}\partial_x^{\beta}n^{\varepsilon}_R \partial_x^{4-\beta}\tilde u -\varepsilon\partial_x^3\mathcal R_1\bigg]=\sum_{i=1}^7A_i^{(3)}.
\end{split}
\end{equation}
Accordingly $I^{(3\times\varepsilon)}$ is decomposed into
\begin{equation}\label{equ76}
\begin{split}
I^{(3\times\varepsilon)}=\sum_{i=1}^7\int\varepsilon\partial_x^3\phi^{\varepsilon}_R A_i^{(3)}=\sum_{i=1}^7I^{(3\times\varepsilon)}_i.
\end{split}
\end{equation}

\emph{Estimate of $I^{(3\times\varepsilon)}_i$ for $3\leq i\leq 7$.} By a direct computation, $I^{(3\times\varepsilon)}_3$ takes the form
\begin{equation}\label{equ75}
\begin{split}
I^{(3\times\varepsilon)}_3=&-\sum_{\beta=1}^{3}C_{3}^{\beta}\int\frac{\varepsilon}{n} \partial_x^3\phi^{\varepsilon}_R\partial_x^{\beta}\tilde n\partial_x^{4-\beta}u^{\varepsilon}_R -\sum_{\beta=1}^3C_{3}^{\beta}\int\frac{\varepsilon^3}{n} \partial_x^3\phi^{\varepsilon}_R\partial_x^{\beta}n^{\varepsilon}_R \partial_x^{4-\beta}u^{\varepsilon}_R.
\end{split}
\end{equation}
The first term on the RHS is bilinear in $(n^{\varepsilon}_R,u^{\varepsilon}_R)$ and is bounded by
\begin{equation*}
\begin{split}
C\varepsilon\|\partial_x^3\phi^{\varepsilon}_R\|^2 +C(\|u^{\varepsilon}_R\|_{H^2}^2 +\varepsilon\|\partial_x^{3}u^{\varepsilon}_R\|^2).
\end{split}
\end{equation*}
For the second term on the RHS of \eqref{equ75}, when $\beta=1,2$, it is bounded by Lemma \ref{L1}
\begin{equation*}
\begin{split}
\int\frac{\varepsilon^3}{n} \partial_x^3\phi^{\varepsilon}_R\partial_x^{\beta}n^{\varepsilon}_R \partial_x^{4-\beta}u^{\varepsilon}_R \leq & C\varepsilon^2\|\partial_x^{\beta}n^{\varepsilon}_R\|^2 (\varepsilon^2\|\partial_x^3\phi^{\varepsilon}_R\|_{L^{\infty}}^2) +C\varepsilon^2\|\partial_x^{4-\beta}u^{\varepsilon}_R\|^2\\
\leq & C_1(1+\varepsilon^2|\!|\!|(u^{\varepsilon}_R,\phi^{\varepsilon}_R) |\!|\!|^2_{\varepsilon})(\varepsilon\|u^{\varepsilon}_R\|_{H^3}^2 +\varepsilon^2|\phi^{\varepsilon}_R|_{H^4}^2).
\end{split}
\end{equation*}
When $\beta=3$, by integration by parts, $H^1\hookrightarrow L^{\infty}$ and Lemma \ref{L2},
\begin{equation*}
\begin{split}
-\int&\frac{\varepsilon^3}{n} \partial_x^3\phi^{\varepsilon}_R\partial_x^{3}n^{\varepsilon}_R \partial_xu^{\varepsilon}_R =\int\frac{\varepsilon^3}{n}\partial_x^4\phi^{\varepsilon}_R \partial_x^2n^{\varepsilon}_R\partial_xu^{\varepsilon}_R\\
&+\int\frac{\varepsilon^3}{n}\partial_x^3\phi^{\varepsilon}_R \partial_x^2n^{\varepsilon}_R\partial_x^2u^{\varepsilon}_R +\int\partial_x\left[\frac{\varepsilon^3}{n}\right]\partial_x^3\phi^{\varepsilon}_R \partial_x^2n^{\varepsilon}_R\partial_xu^{\varepsilon}_R\\
\leq & C(\varepsilon\|\partial_x^3\phi^{\varepsilon}_R \|_{H^1}) (\varepsilon\|\partial_x^2n^{\varepsilon}_R\|) (\varepsilon\|\partial_xu^{\varepsilon}_R\|_{H^{1,\infty}}) (1+\varepsilon^3\|\partial_xn^{\varepsilon}_R\|_{L^{\infty}})\\
\leq & C_1(1+\varepsilon^2|\!|\!|(u^{\varepsilon}_R,\phi^{\varepsilon}_R) |\!|\!|^2_{\varepsilon})(\varepsilon\|\partial_xu^{\varepsilon}_R\|_{H^{2}}^2 +\varepsilon^2\|\phi^{\varepsilon}_R \|_{H^4}^2).
\end{split}
\end{equation*}
This completes the estimate of $I^{(3\times\varepsilon)}_{3}$. The terms $I^{(3\times\varepsilon)}_{i}$ for $i=4,5,6,7$ can be bounded similarly with the same bound.

In summary, we have
\begin{equation*}
\begin{split}
\sum_{i=3}^7I^{(3\times\varepsilon)}_{i}\leq  C_1(1+\varepsilon^2|\!|\!|(u^{\varepsilon}_R,\phi^{\varepsilon}_R) |\!|\!|^2_{\varepsilon})(1+|\!|\!|(u^{\varepsilon}_R,\phi^{\varepsilon}_R) |\!|\!|^2_{\varepsilon}).
\end{split}
\end{equation*}

Proposition \ref{L7} then follows from the following Lemma \ref{L16}, Lemma \ref{L17} and Proposition \ref{PB}.
\end{proof}

\begin{lemma}[\textbf{\emph{{Estimate for $I^{(3\times\varepsilon)}_1$}}}]\label{L16} Let $(n^{\varepsilon}_R,u^{\varepsilon}_R,\phi^{\varepsilon}_R)$ be a solution to \eqref{rem2}, then
\begin{equation*}
\begin{split}
I^{(3\times\varepsilon)}_1 \leq C_1(1+\varepsilon^2|\!|\!|(u^{\varepsilon}_R,\phi^{\varepsilon}_R) |\!|\!|_{\varepsilon}^2)(1+|\!|\!|(u^{\varepsilon}_R,\phi^{\varepsilon}_R) |\!|\!|_{\varepsilon}^2).
\end{split}
\end{equation*}
\end{lemma}
\begin{proof}
Recall from \eqref{equ76},
\begin{equation*}
\begin{split}
I^{(3\times\varepsilon)}_1=\int\partial_{x}^3\phi^{\varepsilon}_R \left[\frac{(1-u)}{n}\partial_{x}^4n^{\varepsilon}_R\right].
\end{split}
\end{equation*}
Taking $\partial_x^4$ of \eqref{rem2-3}, we have
\begin{equation*}
\begin{split}
\partial_{x}^4n^{\varepsilon}_R=\partial_{x}^4\left\{\phi^{\varepsilon}_R -\varepsilon\partial_x^2\phi^{\varepsilon}_R +\varepsilon(\phi^{(1)}\phi^{\varepsilon}_R) +\varepsilon^2\mathcal R_3\right\}.
\end{split}
\end{equation*}
Accordingly, we split $I^{(3\times\varepsilon)}_1$ as
\begin{equation}\label{equ77}
\begin{split}
I^{(3\times\varepsilon)}_1=\int\partial_{x}^3\phi^{\varepsilon}_R \left[\frac{(1-u)}{n}\partial_{x}^4n^{\varepsilon}_R\right] =\sum_{i=1}^4I^{(3\times\varepsilon)}_{1i}.
\end{split}
\end{equation}

\emph{Estimate of $I^{(3\times\varepsilon)}_{11}$ in \eqref{equ77}.} By integration by parts, we have
\begin{equation}\label{equ78}
\begin{split}
I^{(3\times\varepsilon)}_{11}=&\int\partial_{x}^3\phi^{\varepsilon}_R  \left[\frac{(1-u)}{n}\partial_{x}^4\phi^{\varepsilon}_R\right]\\ =&-\int\partial_{x}\left[\frac{(1-u)}{n}\right] \|\partial_{x}^3\phi^{\varepsilon}_R\|^2\\
\leq & C\varepsilon\|\partial_{x}^3\phi^{\varepsilon}_R\|^2 +C\varepsilon^2(\|\partial_xn^{\varepsilon}_R\|_{L^{\infty}} +\|\partial_xu^{\varepsilon}_R\|_{L^{\infty}}) (\varepsilon\|\partial_{x}^3\phi^{\varepsilon}_R\|^2)\\
\leq & C_1(1+\varepsilon^2|\!|\!|(u^{\varepsilon}_R,\phi^{\varepsilon}_R) |\!|\!|^2_{\varepsilon})(\varepsilon\|\partial_{x}^3\phi^{\varepsilon}_R\|^2).
\end{split}
\end{equation}

\emph{Estimate of $I^{(3\times\varepsilon)}_{12}$ in \eqref{equ77}.} By integration by parts twice, we have
\begin{equation}\label{equ79}
\begin{split}
I^{(3\times\varepsilon)}_{12}=&-\int\varepsilon\partial_{x}^3\phi^{\varepsilon}_R  \left[\frac{(1-u)}{n}\partial_{x}^6\phi^{\varepsilon}_R\right]\\
=&-\frac32\int\varepsilon\partial_x \left[\frac{(1-u)}{n}\right]|\partial_{x}^4\phi^{\varepsilon}_R|^2 -\int\varepsilon\partial_{x}^3\phi^{\varepsilon}_R  \partial_{x}^2\left[\frac{(1-u)}{n}\right]\partial_{x}^4\phi^{\varepsilon}_R\\ =&:I^{(3\times\varepsilon)}_{121}+I^{(3\times\varepsilon)}_{122}.
\end{split}
\end{equation}
For the first term $I^{(3\times\varepsilon)}_{121}$, since
\begin{equation*}
\begin{split}
\partial_x(\frac{(1-u)}{n})\leq C\varepsilon+C\varepsilon^3(|\partial_xu^{\varepsilon}_R| +\|\partial_xn^{\varepsilon}_R\|_{L^{\infty}}),
\end{split}
\end{equation*}
by H\"older inequality, Sobolev embedding and Lemma \ref{L2}, we deduce
\begin{equation}\label{equ84}
\begin{split}
I^{(3\times\varepsilon)}_{121} \leq & C\varepsilon^2\|\partial_{x}^4\phi^{\varepsilon}_R\|^2 +C\varepsilon^2(\|\partial_xu^{\varepsilon}_R\|_{L^{\infty}} +\|\partial_xn^{\varepsilon}_R\|_{L^{\infty}}) (\varepsilon^2\|\partial_{x}^4\phi^{\varepsilon}_R\|^2)\\
\leq & C(1+\varepsilon^2|\!|\!|(u^{\varepsilon}_R,\phi^{\varepsilon}_R) |\!|\!|^2_{\varepsilon})(\varepsilon^2\|\partial_{x}^4\phi^{\varepsilon}_R\|^2).
\end{split}
\end{equation}
We note that
\begin{equation}\label{equ80}
\begin{split}
\left|\partial_x^2\left[\frac{(1-u)}{n}\right]\right|\leq & C\Big(\varepsilon +\varepsilon^3(|\partial_x^2n^{\varepsilon}_R|+|\partial_x^2u^{\varepsilon}_R|)\\
& +\varepsilon^4(|\partial_xn^{\varepsilon}_R|+|\partial_xu^{\varepsilon}_R|) +\varepsilon^6(|\partial_xn^{\varepsilon}_R|^2+|\partial_xu^{\varepsilon}_R|^2)\Big).
\end{split}
\end{equation}
To estimate $I^{(3\times\varepsilon)}_{122}$ in \eqref{equ79}, we first observe
\begin{equation}\label{equ81}
\begin{split}
\int\varepsilon^2|\partial_{x}^3\phi^{\varepsilon}_R| |\partial_{x}^4\phi^{\varepsilon}_R|\leq C(\varepsilon\|\partial_{x}^3\phi^{\varepsilon}_R\|^2 +\varepsilon^2\|\partial_{x}^4\phi^{\varepsilon}_R\|^2).
\end{split}
\end{equation}
Secondly, by Sobolev embedding, and Lemma \ref{L1}
\begin{equation}\label{equ82}
\begin{split}
\int\varepsilon^4\partial_{x}^3&\phi^{\varepsilon}_R (|\partial_x^2u^{\varepsilon}_R|+|\partial_x^2n^{\varepsilon}_R|) \partial_{x}^4\phi^{\varepsilon}_R\\
\leq & C\varepsilon^2 (\|\partial_x^2u^{\varepsilon}_R\|+\|\partial_x^2n^{\varepsilon}_R\|) (\varepsilon\|\partial_{x}^3\phi^{\varepsilon}_R\|_{L^{\infty}}) (\varepsilon\|\partial_{x}^4\phi^{\varepsilon}_R\|)\\
\leq & C_1(1+\varepsilon|\!|\!|(u^{\varepsilon}_R,\phi^{\varepsilon}_R) |\!|\!|_{\varepsilon})(\varepsilon^2\|\phi^{\varepsilon}_R\|_{H^4}^2).
\end{split}
\end{equation}
Furthermore, by Sobolev embedding, and Lemma \ref{L1}
\begin{equation*}
\begin{split}
\|\partial_xn^{\varepsilon}_R\|_{L^{\infty}}^2 +\|\partial_xu^{\varepsilon}_R\|_{L^{\infty}}^2\leq & C(\|\partial_xn^{\varepsilon}_R\|_{H^1}^2+\|\partial_xu^{\varepsilon}_R\|_{H^1}^2)\\
\leq & C_1(1+|\!|\!|(u^{\varepsilon}_R,\phi^{\varepsilon}_R) |\!|\!|_{\varepsilon}^2),
\end{split}
\end{equation*}
it is easy to bound
\begin{equation}\label{equ83}
\begin{split}
\int\varepsilon^5\partial_{x}^3\phi^{\varepsilon}_R & \big((|\partial_xn^{\varepsilon}_R|+|\partial_xu^{\varepsilon}_R|) +\varepsilon^2(|\partial_xn^{\varepsilon}_R|^2+|\partial_xu^{\varepsilon}_R|^2)\big) \partial_{x}^4\phi^{\varepsilon}_R\\
\leq & C_1(1+\varepsilon^2|\!|\!|(u^{\varepsilon}_R,\phi^{\varepsilon}_R) |\!|\!|_{\varepsilon}^2)(\varepsilon^2\|\phi^{\varepsilon}_R\|_{H^4}^2).
\end{split}
\end{equation}
By \eqref{equ80}-\eqref{equ83}, the term $I^{(3\times\varepsilon)}_{122}$ in \eqref{equ79} is bounded by
\begin{equation}\label{equ85}
\begin{split}
I^{(3\times\varepsilon)}_{12} \leq  & C_1(1+\varepsilon^2|\!|\!|(u^{\varepsilon}_R,\phi^{\varepsilon}_R) |\!|\!|_{\varepsilon}^2)(1+\varepsilon\|\phi^{\varepsilon}_R\|_{H^3}^2 +\varepsilon^2\|\phi^{\varepsilon}_R\|_{H^4}^2).
\end{split}
\end{equation}
By \eqref{equ84} and \eqref{equ85}, we have
\begin{equation*}
\begin{split}
I^{(3\times\varepsilon)}_{12} \leq  & C_1(1+\varepsilon^2|\!|\!|(u^{\varepsilon}_R,\phi^{\varepsilon}_R) |\!|\!|_{\varepsilon}^2)(1+|\!|\!|(u^{\varepsilon}_R,\phi^{\varepsilon}_R) |\!|\!|_{\varepsilon}^2).
\end{split}
\end{equation*}

\emph{Estimate of $I^{(3\times\varepsilon)}_{13}$ in \eqref{equ77}.} It is bounded similarly to $I^{(3\times\varepsilon)}_{11}$ in \eqref{equ79},
\begin{equation*}
\begin{split}
I^{(3\times\varepsilon)}_{13}
\leq & C_1(1+\varepsilon^2|\!|\!|(u^{\varepsilon}_R,\phi^{\varepsilon}_R) |\!|\!|^2_{\varepsilon})(\|\phi^{\varepsilon}_R\|_{H^2}^2+\varepsilon\|\partial_{x}^3\phi^{\varepsilon}_R\|^2).
\end{split}
\end{equation*}

\emph{Estimate of $I^{(3\times\varepsilon)}_{14}$ in \eqref{equ77}.} Recall that $|1-u|/n$ is uniformly bounded and from \eqref{expan},
\begin{equation*}
\begin{split}
\left|\partial_x[\frac{(1-u)}{n}]\right| \leq C\varepsilon(1+\varepsilon^2(\|\partial_xn^{\varepsilon}_R\|_{L^{\infty}} +\|\partial_xu^{\varepsilon}_R\|_{L^{\infty}})).
\end{split}
\end{equation*}
By using Lemma \ref{L8}, we then have
\begin{equation*}
\begin{split}
I^{(3\times\varepsilon)}_{14}=&\int\partial_{x}^3\phi^{\varepsilon}_R  \left[\frac{\varepsilon^2(1-u)}{n}\partial_{x}^4\mathcal R_3\right]\\
=&-\int\partial_{x}^4\phi^{\varepsilon}_R \left[\frac{\varepsilon^2(1-u)}{n}\right]\partial_{x}^3\mathcal R_3 -\int\partial_{x}^3\phi^{\varepsilon}_R \partial_x\left[\frac{\varepsilon^2(1-u)}{n}\right]\partial_{x}^3\mathcal R_3\\
\leq & C(\|\phi^{(i)}\|_{H^{\tilde s_i}},\varepsilon\|\phi^{\varepsilon}_R\|_{H^3}) (\varepsilon^2\|\phi^{\varepsilon}_R\|_{H^3}^2)\\ &+C(1+\varepsilon^2\|(\partial_xn^{\varepsilon}_R, \partial_xu^{\varepsilon}_R)\|_{L^{\infty}}^2) (\varepsilon\|\partial_x^3\phi^{\varepsilon}_R\|^2 +\varepsilon^2\|\partial_x^4\phi^{\varepsilon}_R\|^2)\\
\leq & C_1(1+\varepsilon^2|\!|\!|(u^{\varepsilon}_R,\phi^{\varepsilon}_R) |\!|\!|_{\varepsilon}^2)(1+|\!|\!|(u^{\varepsilon}_R,\phi^{\varepsilon}_R) |\!|\!|_{\varepsilon}^2).
\end{split}
\end{equation*}
By combining the estimates for $I^{(3\times\varepsilon)}_{1i}(1\leq i\leq 4)$  together, we complete the proof of Lemma \ref{L16}.
\end{proof}

\begin{lemma}[\textbf{\emph{{Estimate for $I^{(3\times\varepsilon)}_2$}}}]\label{L17}
Let $(n^{\varepsilon}_R,u^{\varepsilon}_R,\phi^{\varepsilon}_R)$ be a solution to \eqref{rem2}, then
\begin{equation*}
\begin{split}
I^{(3\times\varepsilon)}_2\leq &-\frac12\frac{d}{dt}\bigg[\bigg(\int\frac{\varepsilon(1+\varepsilon\phi^{(1)})}{n} |\partial_x^3\phi^{\varepsilon}_R|^2 +\int\frac{\varepsilon^2}{n}|\partial_x^4\phi^{\varepsilon}_R|^2\bigg) \bigg] \\
\leq & C_1(1+\varepsilon^2|\!|\!|(u^{\varepsilon}_R,\phi^{\varepsilon}_R) |\!|\!|_{\varepsilon}^2)(1+|\!|\!|(u^{\varepsilon}_R,\phi^{\varepsilon}_R) |\!|\!|^2_{\varepsilon})+\mathcal B^{(3\times\varepsilon)},
\end{split}
\end{equation*}
where
\begin{equation*}
\begin{split}
\mathcal B^{(3\times\varepsilon)}=-\int\partial_x^3\phi^{\varepsilon}_R\partial_x\left[\frac{\varepsilon^2}{n}\right] \partial_{t}\partial_x^4\phi^{\varepsilon}_R.
\end{split}
\end{equation*}
\end{lemma}
\begin{proof}
We first recall that from \eqref{equ76}
\begin{equation*}
\begin{split}
I^{(3\times\varepsilon)}_2=&-\int\frac{\varepsilon}{n}\partial_x^3\phi^{\varepsilon}_R \partial_{t}\partial_x^3n^{\varepsilon}_R.
\end{split}
\end{equation*}
Taking $\partial_{t}\partial_x^3$ of \eqref{rem2-3}, and then inserting the result in $I^{(3\times\varepsilon)}_2$, we have
\begin{equation}\label{equ86}
\begin{split}
I^{(3\times\varepsilon)}_2=&-\int\frac{\varepsilon}{n}\partial_x^3\phi^{\varepsilon}_R \partial_{t}\partial_x^3 \left[\phi^{\varepsilon}_R-\varepsilon\partial_x^2\phi^{\varepsilon}_R +\varepsilon(\phi^{(1)}\phi^{\varepsilon}_R) +\varepsilon^2\mathcal R_3\right] =:\sum_{i=1}^4I^{(3\times\varepsilon)}_{2i}.
\end{split}
\end{equation}

\emph{Estimate of $I^{(3\times\varepsilon)}_{21}$ in \eqref{equ86}.} By integration by parts in $t$, and then using Sobolev embedding $H^1\hookrightarrow L^{\infty}$ and Lemma \ref{L2}, we have
\begin{equation}\label{equ87}
\begin{split}
I^{(3\times\varepsilon)}_{21}
=&-\int\frac{\varepsilon}{n}\partial_x^3\phi^{\varepsilon}_R \partial_{t}\partial_x^3\phi^{\varepsilon}_R\\
=&-\frac12\frac{d}{dt}\int\frac{\varepsilon}{n}|\partial_x^3\phi^{\varepsilon}_R|^2 +\frac12\int\partial_t\left[\frac{\varepsilon}{n}\right]|\partial_x^3\phi^{\varepsilon}_R|^2\\
=&-\frac12\frac{d}{dt}\int\frac{\varepsilon}{n}|\partial_x^3\phi^{\varepsilon}_R|^2 -\int\varepsilon\left[\frac{\varepsilon\partial_t\tilde n}{n^2}+\frac{\varepsilon^3\partial_tn^{\varepsilon}_R}{n^2}\right] |\partial_x^3\phi^{\varepsilon}_R|^2\\
\leq & -\frac12\frac{d}{dt}\int\frac{\varepsilon}{n}|\partial_x^3\phi^{\varepsilon}_R|^2 +C\varepsilon(1+\|\varepsilon\partial_tn^{\varepsilon}_R\|_{L^{\infty}}) (\varepsilon\|\partial_x^3\phi^{\varepsilon}_R\|^2)\\
\leq & -\frac12\frac{d}{dt}\int\frac{\varepsilon}{n}|\partial_x^3\phi^{\varepsilon}_R|^2 +C_1\varepsilon(1+\varepsilon|\!|\!|(u^{\varepsilon}_R,\phi^{\varepsilon}_R) |\!|\!|_{\varepsilon}) (\varepsilon\|\partial_x^3\phi^{\varepsilon}_R\|^2).
\end{split}
\end{equation}

\emph{Estimate of $I^{(3\times\varepsilon)}_{22}$ in \eqref{equ86}.} By integration by parts, we have
\begin{equation}\label{bad3}
\begin{split}
I^{(3\times\varepsilon)}_{22}=&\int\partial_x^3\phi^{\varepsilon}_R \left[\frac{\varepsilon^2}{n}\right] \partial_{t}\partial_x^5\phi^{\varepsilon}_R\\
=&\underbrace{-\int\partial_x^4\phi^{\varepsilon}_R\left[\frac{\varepsilon^2}{n}\right] \partial_{t}\partial_x^4\phi^{\varepsilon}_R}_{I^{(3\times\varepsilon)}_{221}} \underbrace{-\int\partial_x^3\phi^{\varepsilon}_R\partial_x\left[\frac{\varepsilon^2}{n}\right] \partial_{t}\partial_x^4\phi^{\varepsilon}_R}_{\mathcal B^{(3\times\varepsilon)}}.
\end{split}
\end{equation}
For the first term $I^{(3\times\varepsilon)}_{221}$ in \eqref{bad3}, we have
\begin{equation*}
\begin{split}
I^{(3\times\varepsilon)}_{221}=&-\frac12\frac{d}{dt}\int\left[\frac{\varepsilon^2}{n}\right] |\partial_x^4\phi^{\varepsilon}_R|^2 +\frac12\int\partial_t\left[\frac{\varepsilon^2}{n}\right] |\partial_x^4\phi^{\varepsilon}_R|^2\\
=&-\frac12\frac{d}{dt}\int\left[\frac{\varepsilon^2}{n}\right]|\partial_x^4\phi^{\varepsilon}_R|^2 -\frac12\int\varepsilon^2\left[\frac{\varepsilon\partial_t\tilde n}{n^2}+\frac{\varepsilon^3\partial_tn^{\varepsilon}_R}{n^2}\right] |\partial_x^4\phi^{\varepsilon}_R|^2\\
\leq & -\frac12\frac{d}{dt}\int\left[\frac{\varepsilon^2}{n}\right]|\partial_x^4\phi^{\varepsilon}_R|^2 +C_1\varepsilon(1+\varepsilon|\!|\!|(u^{\varepsilon}_R,\phi^{\varepsilon}_R) |\!|\!|_{\varepsilon}) (\varepsilon^2\|\partial_x^4\phi^{\varepsilon}_R\|^2).
\end{split}
\end{equation*}
The term $\mathcal B^{(3\times\varepsilon)}$ cannot be controlled in terms of $|\!|\!|(u^{\varepsilon}_R,\phi^{\varepsilon}_R)|\!|\!|_{\varepsilon}$ so far (see Remark \ref{remark}). Its estimate is postponed to Section \ref{set3.3} by an exact cancellation with $\mathcal B^{(2)}$ in Corollary \ref{Cor}.

\emph{Estimate of $I^{(3\times\varepsilon)}_{23}$ in \eqref{equ86}.} Similar to the estimate of $I^{(3\times\varepsilon)}_{21}$ in \eqref{equ87}, we have
\begin{equation*}
\begin{split}
I^{(3\times\varepsilon)}_{23} \leq -\frac12\frac{d}{dt}\int \frac{\varepsilon^2\phi^{(1)}}{n}|\partial_x^3\phi^{\varepsilon}_R|^2 +C_1\varepsilon(1+\varepsilon|\!|\!|(u^{\varepsilon}_R,\phi^{\varepsilon}_R) |\!|\!|_{\varepsilon}) (1+|\!|\!|(u^{\varepsilon}_R,\phi^{\varepsilon}_R) |\!|\!|_{\varepsilon}).
\end{split}
\end{equation*}

\emph{Estimate of $I^{(3\times\varepsilon)}_{24}$ in \eqref{equ86}.} Integration by parts yields
\begin{equation}\label{equ88}
\begin{split}
I^{(3\times\varepsilon)}_{24}=&-\int\partial_x^3\phi^{\varepsilon}_R \left[\frac{\varepsilon^3}{n}\right]\partial_t\partial_x^3\mathcal R_3\\
=&\int\partial_x^4\phi^{\varepsilon}_R \left[\frac{\varepsilon^3}{n}\right]\partial_t\partial_x^2\mathcal R_3 +\int\partial_x^3\phi^{\varepsilon}_R \partial_x\left[\frac{\varepsilon^3}{n}\right]\partial_t\partial_x^2\mathcal R_3.
\end{split}
\end{equation}
By Lemma \ref{L8} in the Appendix,
\begin{equation*}
\begin{split}
\varepsilon\|\varepsilon\partial_t\partial_x^2\mathcal R_3\|^2\leq & \varepsilon C(\|\phi^{(i)}\|_{H^{\tilde s_i}},\varepsilon\|\phi^{\varepsilon}_R\|_{H^2}) \|\varepsilon\partial_t\phi^{\varepsilon}_R\|_{H^2}
\end{split}
\end{equation*}
and by Sobolev embedding and Lemma \ref{L1}
\begin{equation*}
\begin{split}
\|\partial_x(\frac{1}{n})\|_{L^{\infty}}^2\leq & C_1\varepsilon(1+\varepsilon^2|\!|\!|(u^{\varepsilon}_R,\phi^{\varepsilon}_R) |\!|\!|_{\varepsilon}^2).
\end{split}
\end{equation*}
By \eqref{equ88}, we therefore have
\begin{equation}\label{equ96}
\begin{split}
I^{(3\times\varepsilon)}_{24} \leq & C(1+\varepsilon^2|\!|\!|(u^{\varepsilon}_R,\phi^{\varepsilon}_R) |\!|\!|_{\varepsilon}^2)(\varepsilon^2\|\partial_x^3\phi^{\varepsilon}_R\|_{H^1}^2)+ C_1\varepsilon (\varepsilon\|\varepsilon\partial_t\phi^{\varepsilon}_R\|_{H^2}^2)\\
\leq & C_1(1+\varepsilon^2|\!|\!|(u^{\varepsilon}_R,\phi^{\varepsilon}_R) |\!|\!|_{\varepsilon}^2)(1+|\!|\!|(u^{\varepsilon}_R,\phi^{\varepsilon}_R) |\!|\!|^2_{\varepsilon}).
\end{split}
\end{equation}
Lemma \ref{L17} then follows.
\end{proof}

\begin{remark}\label{remark}
By Lemma \ref{L3}, only $\|\partial_t\partial_x^2\phi^{\varepsilon}_R\|_{L^2}$ can be controlled in terms of $|\!|\!|(u^{\varepsilon}_R,\phi^{\varepsilon}_R)|\!|\!|_{\varepsilon}^2$ through $\|\partial_{t}\phi^{\varepsilon}_R\|_{H^1}^2$ by Lemma \ref{L2}. However, upon integration by parts, there will be a contribution $\int\partial_x^5\phi^{\varepsilon}_R\partial_x\left[\frac{\varepsilon^2}{n}\right] \partial_{t}\partial_x^2\phi^{\varepsilon}_R$, which is not controllable in terms of $|\!|\!|(u^{\varepsilon}_R,\phi^{\varepsilon}_R)|\!|\!|_{\varepsilon}^2$ due to $\partial_x^5\phi^{\varepsilon}_R$.
\end{remark}

However, $\mathcal B^{(3\times\varepsilon)}$ is controlled by an exact cancellation by using \eqref{rem2-3} one more time. Besides the term $\mathcal B^{(3\times\varepsilon)}$, there is a term $\mathcal B^{(2)}$ with the same structure in Corollary \ref{Cor}. Recalling $\mathcal B^{(2)}$ in Corollary \ref{Cor}, we obtain
\begin{equation}\label{B23}
\begin{split}
\mathcal G^{(2,\varepsilon)}=&\mathcal B^{(2)}+\mathcal B^{(3\times\varepsilon)} =\int\partial_x(\frac{\varepsilon}n)\partial_{x}^3\phi^{\varepsilon}_R
\left[\partial_{t}\partial_x^2(\phi^{\varepsilon}_R -\varepsilon\partial_x^2\phi^{\varepsilon}_R)\right].
\end{split}
\end{equation}
The crucial observation is that the combination $(\phi^{\varepsilon}_R -\varepsilon\partial_x^2\phi^{\varepsilon}_R)$ exactly appears in the Poisson equation \eqref{rem2-3} and can be controlled.

\subsection{Control of $\mathcal G^{(2,\varepsilon)}$}\label{set3.3}
\begin{proposition}\label{PB}
Let $(n^{\varepsilon}_R,u^{\varepsilon}_R,\phi^{\varepsilon}_R)$ be a solution to \eqref{rem2}, then
\begin{equation}\label{Gron}
\begin{split}
\mathcal G^{(2,\varepsilon)} \leq & C_1(1+\varepsilon^2|\!|\!|(u^{\varepsilon}_R,\phi^{\varepsilon}_R) |\!|\!|_{\varepsilon}^2)(1+|\!|\!|(u^{\varepsilon}_R,\phi^{\varepsilon}_R) |\!|\!|^2_{\varepsilon}),
\end{split}
\end{equation}
where $|\!|\!|(u^{\varepsilon}_R,\phi^{\varepsilon}_R) |\!|\!|_{\varepsilon}$ is defined in \eqref{def-A}.
\end{proposition}
\begin{proof}
Recall \eqref{B23}. From the Poisson equation \eqref{rem2-3}, we have
\begin{equation}\label{equ90}
\begin{split}
\mathcal G^{(2,\varepsilon)} =&\int\partial_x(\frac{\varepsilon}n)\partial_{x}^3\phi^{\varepsilon}_R \left[\partial_{t}\partial_x^2(\phi^{\varepsilon}_R -\varepsilon\partial_x^2\phi^{\varepsilon}_R)\right]\\
=&\int\partial_x(\frac{\varepsilon}n)\partial_{x}^3\phi^{\varepsilon}_R
\left[\partial_{t}\partial_x^2(n^{\varepsilon}_R-\varepsilon(\phi^{(1)}\phi^{\varepsilon}_R) -\varepsilon^2\mathcal R_3)\right]=\sum_{i=1}^3\mathcal G^{(2,\varepsilon)}_i.
\end{split}
\end{equation}

\emph{Estimate of $\mathcal G^{(2,\varepsilon)}_1$.} By integration by parts, we have
\begin{equation}\label{equ89}
\begin{split}
\mathcal G^{(2,\varepsilon)}_1=&\int\partial_{x}^3\phi^{\varepsilon}_R \partial_x\left[\frac{\varepsilon}n\right]\partial_{t}\partial_x^2n^{\varepsilon}_R\\
=&-\int\partial_{x}^4\phi^{\varepsilon}_R \partial_x\left[\frac{\varepsilon}n\right]\partial_{tx}n^{\varepsilon}_R -\int\partial_{x}^3\phi^{\varepsilon}_R \partial_x^2\left[\frac{\varepsilon}n\right]\partial_{tx}n^{\varepsilon}_R\\
=&:\mathcal G^{(2,\varepsilon)}_{11}+\mathcal G^{(2,\varepsilon)}_{12}.
\end{split}
\end{equation}
Recalling the expression of $n$ in \eqref{expan}, we have
\begin{equation*}
\begin{split}
|\partial_x(\frac{\varepsilon}{n})|\leq C(\varepsilon^2+\varepsilon^4|\partial_xn^{\varepsilon}_R|),
\end{split}
\end{equation*}
and
\begin{equation*}
\begin{split}
\left|\partial_x^2(\frac{\varepsilon}{n})\right|\leq C(\varepsilon^2 +\varepsilon^4|\partial_x^2n^{\varepsilon}_R| +\varepsilon^5|\partial_xn^{\varepsilon}_R| +\varepsilon^7|\partial_xn^{\varepsilon}_R|^2).
\end{split}
\end{equation*}
The first term $\mathcal G^{(2,\varepsilon)}_{11}$ in \eqref{equ89} is bounded by Sobolev embedding, Lemma \ref{L2} and \ref{L1}
\begin{equation}\label{equ95}
\begin{split}
\mathcal G^{(2,\varepsilon)}_{11} \leq & C(\varepsilon\|\partial_x^4\phi^{\varepsilon}_R\|) \|\varepsilon\partial_{tx}n^{\varepsilon}_R\| +C\varepsilon(\varepsilon\|\partial_x^4\phi^{\varepsilon}_R\|) (\varepsilon\|\partial_xn^{\varepsilon}_R\|_{L^{\infty}}) \|\varepsilon\partial_{tx}n^{\varepsilon}_R\|\\
\leq & C\varepsilon^2\|\partial_x^4\phi^{\varepsilon}_R\|^2 +C_1(1+\varepsilon^2|\!|\!|(u^{\varepsilon}_R,\phi^{\varepsilon}_R) |\!|\!|_{\varepsilon}^2) (1+|\!|\!|(u^{\varepsilon}_R,\phi^{\varepsilon}_R) |\!|\!|_{\varepsilon}^2).
\end{split}
\end{equation}
To estimate $\mathcal G^{(2,\varepsilon)}_{12}$ in \eqref{equ90}, we first observe that by Lemma \ref{L2} and Lemma \ref{L1},
\begin{equation}\label{equ91}
\begin{split}
\varepsilon^2\int|\partial_{x}^3\phi^{\varepsilon}_R\partial_{tx}n^{\varepsilon}_R| \leq & C\varepsilon\|\partial_{x}^3\phi^{\varepsilon}_R\|^2 +C\|\varepsilon\partial_{tx}n^{\varepsilon}_R\|^2\\
\leq & C_1(1+|\!|\!|(u^{\varepsilon}_R,\phi^{\varepsilon}_R) |\!|\!|_{\varepsilon}^2).
\end{split}
\end{equation}
Secondly, by Sobolev embedding $H^1\hookrightarrow L^{\infty}$, Lemma \ref{L1} and Lemma \ref{L2}
\begin{equation}\label{equ92}
\begin{split}
\varepsilon^4\int|\partial_{x}^3\phi^{\varepsilon}_R\partial_{x}^2n^{\varepsilon}_R \partial_{tx}n^{\varepsilon}_R| \leq & C\varepsilon^2\|\varepsilon\partial_{x}^3\phi^{\varepsilon}_R\|_{L^{\infty}}^2 \|\partial_{x}^2n^{\varepsilon}_R\|^2 +C\varepsilon^2\|\varepsilon\partial_{tx}n^{\varepsilon}_R\|^2\\
\leq & C_1(1+\varepsilon^2|\!|\!|(u^{\varepsilon}_R,\phi^{\varepsilon}_R) |\!|\!|_{\varepsilon}^2)(1+\varepsilon^2\|\phi^{\varepsilon}_R\|_{H^4}^2 ).
\end{split}
\end{equation}
Finally,
\begin{equation}\label{equ93}
\begin{split}
\varepsilon^5\int&|\partial_{x}^3\phi^{\varepsilon}_R| (1+|\partial_{x}n^{\varepsilon}_R|)|\partial_{x}n^{\varepsilon}_R| |\partial_{tx}n^{\varepsilon}_R|\\
\leq & C\varepsilon^2(1+\|\partial_{x}n^{\varepsilon}_R\|_{L^{\infty}}^2) \|\varepsilon\partial_{x}^3\phi^{\varepsilon}_R\|^2 +C\varepsilon^2\|\partial_{x}n^{\varepsilon}_R\|_{L^{\infty}}^2 \|\varepsilon\partial_{tx}n^{\varepsilon}_R\|^2\\
\leq & C_1(1+\varepsilon^2|\!|\!|(u^{\varepsilon}_R,\phi^{\varepsilon}_R) |\!|\!|_{\varepsilon}^2)(1+|\!|\!|(u^{\varepsilon}_R,\phi^{\varepsilon}_R) |\!|\!|_{\varepsilon}^2).
\end{split}
\end{equation}
Summarizing inequalities \eqref{equ91}, \eqref{equ92} and \eqref{equ93}, we have
\begin{equation}\label{equ94}
\begin{split}
\mathcal G^{2,\varepsilon}_{12} \leq & C_1(1+\varepsilon^2|\!|\!|(u^{\varepsilon}_R,\phi^{\varepsilon}_R) |\!|\!|_{\varepsilon}^2)(1+|\!|\!|(u^{\varepsilon}_R,\phi^{\varepsilon}_R) |\!|\!|_{\varepsilon}^2).
\end{split}
\end{equation}
Combining \eqref{equ95} and \eqref{equ94}, we can bound $\mathcal G^{2,\varepsilon}_1$ in \eqref{equ89} as
\begin{equation}\label{equ97}
\begin{split}
\mathcal G^{2,\varepsilon}_1 \leq & C_1(1+\varepsilon^2|\!|\!|(u^{\varepsilon}_R,\phi^{\varepsilon}_R) |\!|\!|_{\varepsilon}^2)(1+|\!|\!|(u^{\varepsilon}_R,\phi^{\varepsilon}_R) |\!|\!|_{\varepsilon}^2).
\end{split}
\end{equation}

\emph{Estimate of $\mathcal G^{(2,\varepsilon)}_{2}$ in \eqref{equ90}.} From Lemma \ref{L3} and \ref{L2}
\begin{equation*}
\begin{split}
\varepsilon\|\varepsilon\partial_t\partial_x^2 (\phi^{(1)}\phi^{\varepsilon}_R)\|^2_{L^2}\leq & C_1\|\varepsilon \partial_tn^{\varepsilon}_R\|_{H^1}^2 +C_1\varepsilon\|\phi^{\varepsilon}_R\|_{H^2}^2\\
\leq & C_1(1+|\!|\!|(u^{\varepsilon}_R,\phi^{\varepsilon}_R) |\!|\!|_{\varepsilon}^2),
\end{split}
\end{equation*}
and
\begin{equation*}
\begin{split}
\left|\partial_x(\frac{\varepsilon^2}{n})\right|\leq C\varepsilon^3(1+\varepsilon^2|\partial_xn^{\varepsilon}_R|).
\end{split}
\end{equation*}
By H\"older inequality and Sobolev embedding, we have
\begin{equation}\label{equ98}
\begin{split}
\mathcal G^{(2,\varepsilon)}_{2}=&\int\partial_x(\frac{\varepsilon^2}{n}) \partial_x^3\phi^{\varepsilon}_R\partial_t\partial_x^2(\phi^{(1)}\phi^{\varepsilon}_R)\\
\leq & C\varepsilon(1+\varepsilon^2\|\partial_xn\|_{L^{\infty}}^2)(\varepsilon\|\partial_x^3\phi^{\varepsilon}_R\|^2) +C\varepsilon(\varepsilon\|\varepsilon\partial_t\partial_x^2\phi^{\varepsilon}_R\|^2)\\
\leq & C_1\varepsilon(1+\varepsilon^2|\!|\!|(u^{\varepsilon}_R,\phi^{\varepsilon}_R) |\!|\!|_{\varepsilon}^2)(1+|\!|\!|(u^{\varepsilon}_R,\phi^{\varepsilon}_R) |\!|\!|_{\varepsilon}^2).
\end{split}
\end{equation}

\emph{Estimate of $\mathcal G^{(2,\varepsilon)}_{3}$ in \eqref{equ90}.} As the estimate for $I^{(3\times\varepsilon)}_{24}$ in \eqref{equ96}, we have
\begin{equation}\label{equ99}
\begin{split}
\mathcal G^{(2,\varepsilon)}_{3} \leq & C_1(1+\varepsilon^2|\!|\!|(u^{\varepsilon}_R,\phi^{\varepsilon}_R) |\!|\!|_{\varepsilon}^2)(1+|\!|\!|(u^{\varepsilon}_R,\phi^{\varepsilon}_R) |\!|\!|^2_{\varepsilon}).
\end{split}
\end{equation}
We complete the proof of Proposition \ref{PB} by adding the estimates \eqref{equ97}, \eqref{equ98} and \eqref{equ99} together.
\end{proof}

\begin{proof}[\textbf{Proof of Theorem \ref{thm1} for $T_i=0$}]
Adding the Propositions \ref{L4} with $\gamma=0,1$, Corollary \ref{Cor} and Proposition \ref{L7} and \ref{PB} together, we obtain
\begin{equation}\label{Gron}
\begin{split}
\frac12\frac{d}{dt}&[\|u^{\varepsilon}_R\|_{H^2}^2 +\varepsilon\|\partial_x^3u^{\varepsilon}_R\|_{L^2}^2]
   +\frac12\frac{d}{dt}[(\int\frac{1+\varepsilon\phi^{(1)}}{n} |\phi^{\varepsilon}_R|^2 +\int\frac{\varepsilon}{n}|\partial_x\phi^{\varepsilon}_R|^2)\\
&+(\int\frac{1+\varepsilon\phi^{(1)}}{n}|\partial_x\phi^{\varepsilon}_R|^2 +\int\frac{\varepsilon}{n}|\partial_x^2\phi^{\varepsilon}_R|^2)
   +(\int\frac{1+\varepsilon\phi^{(1)}}{n}|\partial_x^2\phi^{\varepsilon}_R|^2 \\&+\int\frac{\varepsilon}{n}|\partial_x^3\phi^{\varepsilon}_R|^2) +(\int\frac{\varepsilon(1+\varepsilon\phi^{(1)})}{n}|\partial_x^3\phi^{\varepsilon}_R|^2 +\int\frac{\varepsilon^2}{n}|\partial_x^4\phi^{\varepsilon}_R|^2)]\\
\leq & C_1(1+\varepsilon^2|\!|\!|(u^{\varepsilon}_R,\phi^{\varepsilon}_R) |\!|\!|_{\varepsilon}^2)(1+|\!|\!|(u^{\varepsilon}_R,\phi^{\varepsilon}_R) |\!|\!|^2_{\varepsilon}).
\end{split}
\end{equation}
Since $\phi^{(1)}$ is uniformly bounded, there exists some $\varepsilon_1>0$ such that when $\varepsilon<\varepsilon_1$, $1+\varepsilon\phi^{(1)}\geq 1/2$. Integrating the inequality \eqref{Gron} over $(0,t)$ yields
\begin{equation*}
\begin{split}
|\!|\!|(u^{\varepsilon}_R,\phi^{\varepsilon}_R)(t)|\!|\!|^2_{\varepsilon}\leq & C|\!|\!|(u^{\varepsilon}_R,\phi^{\varepsilon}_R)(0)|\!|\!|^2_{\varepsilon} +\int_0^t C_1(1+\varepsilon^2|\!|\!|(u^{\varepsilon}_R,\phi^{\varepsilon}_R) |\!|\!|_{\varepsilon}^2)(1+|\!|\!|(u^{\varepsilon}_R,\phi^{\varepsilon}_R) |\!|\!|^2_{\varepsilon})ds\\
\leq & C|\!|\!|(u^{\varepsilon}_R,\phi^{\varepsilon}_R)(0)|\!|\!|^2_{\varepsilon} +\int_0^tC_1(1+\varepsilon\tilde C)(1+|\!|\!|(u^{\varepsilon}_R,\phi^{\varepsilon}_R) |\!|\!|^2_{\varepsilon})ds,
\end{split}
\end{equation*}
where $C$ is an absolute constant.

Recall that $C_1$ depends on $|\!|\!|(u^{\varepsilon}_R, \phi^{\varepsilon}_R)|\!|\!|^2_{\varepsilon}$ through $\varepsilon|\!|\!|(u^{\varepsilon}_R, \phi^{\varepsilon}_R)|\!|\!|^2_{\varepsilon}$ and is nondecreasing. Let $C_1'=C_1(1)$ and $C_2>C\sup_{\varepsilon<1}|\!|\!|(u^{\varepsilon}_R,\phi^{\varepsilon}_R) (0)|\!|\!|^2_{\varepsilon}$. For any arbitrarily given $\tau>0$, we choose $\tilde C$ sufficiently large such that $\tilde C>e^{4C_1'\tau}(1+C_2)(1+C_1')$. Then there exists $\varepsilon_0>0$ such that ${\varepsilon}\tilde C\leq 1$ for all $\varepsilon<\varepsilon_0$, we have
\begin{equation}\label{e999}
\begin{split}
\sup_{0\leq t\leq\tau}|\!|\!|(u^{\varepsilon}_R, \phi^{\varepsilon}_R)(t)|\!|\!|^2_{\varepsilon}\leq e^{4C_1'\tau}(C_2+1)<\tilde C.
\end{split}
\end{equation}
In particular, we have  the uniform bound for $(u^{\varepsilon}_R,\phi^{\varepsilon}_R)$
\begin{equation*}
\begin{split}
\sup_{0\leq t\leq\tau}\|(u^{\varepsilon}_R,\phi^{\varepsilon}_R)(t)\|_{H^2}^2 +\varepsilon\|\partial_x^3(u^{\varepsilon}_R, \phi^{\varepsilon}_R)(t)\|^2 +\varepsilon^2\|\partial_x^4\phi^{\varepsilon}_R(t)\|^2 \leq \tilde C.
\end{split}
\end{equation*}
On the other hand, by Lemma \ref{L1} and \eqref{e999}, we have
\begin{equation*}
\begin{split}
\sup_{0\leq t\leq\tau}\|n^{\varepsilon}_R(t)\|_{H^2}^2\leq \tilde C.
\end{split}
\end{equation*}
It is now standard to obtain uniform estimates independent of $\varepsilon$ by the continuity method. The proof of Theorem \ref{thm1} is complete for the case of $T_i=0$.
\end{proof}
\begin{proof}[\textbf{Proof of Theorem \ref{thm1}}]
Recall that the case of $T_i>0$ is proved in Section 2. The proof of Theorem \ref{thm1} is complete.
\end{proof}

\appendix
\section{}
This appendix consists of two main parts. In the first one, we give a simple proof of Theorem \ref{thm3}, and in the second one, we derive the remainder system.

\begin{proof}[Proof of Theorem \ref{thm3}]
We only give \emph{a priori} estimate. Consider the equation \eqref{linearized} for $k=2$. Let $\tau>0$ be arbitrarily fixed and $\tilde s_1$ be sufficiently large as in Theorem \ref{thm2}. Note $G^{(1)}$ only depends on $n^{(1)}\in H^{\tilde s_1}$, which is assumed to be bounded in $L^{\infty}(-\tau,\tau;H^{\tilde s_1})$. Multiplying the equation \eqref{linearized} with $n^{(2)}$ and integrating over $\Bbb R$, by integration by parts, the dispersive term vanishes and we have
\begin{equation*}
\begin{split}
\frac12\frac{d}{dt}\|n^{(2)}\|^2\leq & C\|\partial_xn^{(1)}\|_{L^{\infty}}\|n^{(2)}\|^2+C\|G^{(1)}\|\|n^{(2)}\|^2 \leq  C\|n^{(2)}\|^2.
\end{split}
\end{equation*}
Higher order estimate is similar, and we then have a unique global solution for $n^{(2)}$. For $k=3$, recalling that $G^{(2)}$ only depends on $n^{(1)}$ and $n^{(2)}$, we have the similar estimate and hence the global existence and uniqueness is obtained. The general case can be proved by induction.
\end{proof}

\noindent\emph{Derivation of the remainder system \eqref{rem}.} In the following, we derive the remainder system \eqref{rem} for $(n^{\varepsilon}_R,u^{\varepsilon}_R,\phi^{\varepsilon}_R)$. From Theorem \ref{thm2} and \ref{thm3}, we have the following systems:
\begin{equation}\label{equ19}
\begin{split}
&(n^{(1)},u^{(1)},\phi^{(1)}) \text{ satisfies \eqref{equ6} and \eqref{kdv}}\\
&(n^{(2)},u^{(2)},\phi^{(2)}) \text{ satisfies \eqref{equ16} and \eqref{e4}}\\
&(n^{(k)},u^{(k)},\phi^{(k)}) \text{ satisfies \eqref{relation} and \eqref{linearized}},\ \ k=3,4.
\end{split}
\end{equation}
By Theorem \ref{thm2} and \ref{thm3}, the solutions $(n^{(k)},u^{(k)},\phi^{(k)})\in H^{\tilde s_k}\ (k=1,2,3,4)$ are global when $\tilde s_k$ is sufficiently large. However, the systems such as \eqref{e2} and \eqref{e3} are convenient for us to derive the remainder system. Therefore, we first remark that from \eqref{equ19} we can derive \eqref{e2} and higher order counterparts. Indeed, from \eqref{equ16-1}, we have exactly \eqref{e2-1}. Differentiating \eqref{equ16-2} with respect to $x$, we have exactly \eqref{e2-2}. From \eqref{kdv}, we subtract $V$ times \eqref{e2-1} and ${T_e}/{(4\pi e\bar{n}M)}$ times \eqref{e2-3}, we get \eqref{e2-2}. Similarly, we can obtain \eqref{e3} and the system $(\mathcal S_3)$ for the coefficients of $\varepsilon^4$.
\bigskip

\emph{The coefficients of $\varepsilon^4$.} Setting the coefficient of $\varepsilon^4$ to be 0, we obtain
\begin{subequations}\label{e-4}
\begin{numcases}{(\mathcal S_3)}
\ \partial_tn^{(3)}-V\partial_xn^{(4)}+\partial_xu^{(4)} +\partial_x(\sum_{1\leq i,j\leq 3;i+j=4}n^{(i)}u^{(j)})=0\\
\ \partial_tu^{(3)}-V\partial_xu^{(4)}+\sum_{1\leq i,j\leq 4;i+j=4}u^{(i)}\partial_xu^{(j)} +\frac{T_i}{M}\partial_xn^{(4)}\nonumber\\
\ \ \ \ \ \ \ \ \ \ \ \ -\frac{T_i}{M}[\partial_x(n^{(1)}n^{(3)}) +(n^{(2)}-(n^{(1)})^2)\partial_xn^{(2)}\nonumber\\ \ \ \ \ \ \ \ \ \ \ \ \ +((n^{(1)})^3-2n^{(1)}n^{(2)})\partial_xn^{(1)}]=-\frac{e}{M}\partial_x\phi^{(4)}\\
\partial_x^2\phi^{(3)}=4\pi e\bar{n}[\kappa\phi^{(4)}+\frac{\kappa^2}{2!}(2\phi^{(1)}\phi^{(3)} +(\phi^{(2)})^2)\nonumber\\
\ \ \ \ \ \ \ \ \ \ \ \ +\frac{\kappa^3}{3!}(3(\phi^{(1)})^2\phi^{(2)}) +\frac{\kappa^4}{4!}(\phi^{(1)})^4-n^{(4)}],
\end{numcases}
\end{subequations}
where we set $\kappa=e/T_e$ for simplification. Inserting the expansion \eqref{expan} into the system \eqref{equ10}, and then subtracting  \{$\varepsilon\times$\eqref{e1}+$\varepsilon^2\times$\eqref{e2} +$\varepsilon^3\times$\eqref{e3}+$\varepsilon^4\times$\eqref{e-4}\}, we get the remainder system \eqref{rem} for $(n^{\varepsilon}_R, u^{\varepsilon}_R, \phi^{\varepsilon}_R)$. The details of deriving \eqref{rem-3} are given below, while the other two equations \eqref{rem-1} and \eqref{rem-2} are similar and omitted. It is easy to see that the remainder terms on the LHS is $\varepsilon^5\partial_x^2\phi^{(4)}+\varepsilon^4\partial_x^2\phi^{\varepsilon}_R$. By Taylor expansion, we have
\begin{equation}\label{equ20}
\begin{split}
e^{\kappa(\varepsilon\widehat\phi +\varepsilon^3\phi^{\varepsilon}_R)} =&1+\frac1{1!}\kappa(\varepsilon\widehat\phi+\varepsilon^3\phi^{\varepsilon}_R)+\cdots +\frac1{4!}\kappa^4(\varepsilon\widehat\phi+\varepsilon^3\phi^{\varepsilon}_R)^4\\ &+\frac1{4!}\int_0^{1}e^{\theta\kappa({\varepsilon\widehat\phi +\varepsilon^3\phi^{\varepsilon}_R})} (1-\theta)^4(\kappa(\varepsilon\widehat\phi +\varepsilon^3\phi^{\varepsilon}_R))^5d\theta,
\end{split}
\end{equation}
where $\varepsilon\widehat\phi=\varepsilon\phi^{(1)} +\cdots +\varepsilon^4\phi^{(4)}$. Now, the constant $1$ cancels with the $1$ in $n$ of \eqref{expan}. From \eqref{equ6}, the coefficient of the $\varepsilon$ order is also exactly canceled. Then by \eqref{e2},  \eqref{e3} and \eqref{e-4}, all the coefficients before $\varepsilon^0, \varepsilon^1, \varepsilon^2, \varepsilon^3$ and $\varepsilon^4$ vanish except the terms involving $\phi^{\varepsilon}_R$. Therefore, the remainder on the RHS of \eqref{rem-3} is give by
\begin{equation*}
\begin{split}
&4\pi e\bar n\{\kappa\varepsilon^3\phi^{\varepsilon}_R +\frac{\kappa^2}{2!}[\varepsilon^6(\phi^{\varepsilon}_R)^2 +2\varepsilon^4\widehat\phi\phi^{\varepsilon}_R] +\frac{\kappa^3}{3!}[\varepsilon^9(\phi^{\varepsilon}_R)^3 +3\varepsilon^7\widehat\phi(\phi^{\varepsilon}_R)^2 +3\varepsilon^5(\widehat\phi)^2\phi^{\varepsilon}_R]\\
&+\frac{\kappa^4}{4!}[\varepsilon^{12}(\phi^{\varepsilon}_R)^4 +4\varepsilon^{10}\widehat\phi(\phi^{\varepsilon}_R)^3 +6\varepsilon^{8}(\widehat\phi)^2(\phi^{\varepsilon}_R)^2 +4\varepsilon^{6}(\widehat\phi)^3\phi^{\varepsilon}_R] -\varepsilon^3n^{\varepsilon}_R+\varepsilon^5\widehat R(\varepsilon\phi^{\varepsilon}_R)+\varepsilon^5R_1\},
\end{split}
\end{equation*}
where $\widehat R$ is the remainder terms corresponding to the last integral term of \eqref{equ20}, and $R_1=R_1(\phi^{(1)},\phi^{(2)},\phi^{(3)},\phi^{(4)})$  involving only $(\phi^{(1)},\phi^{(2)},\phi^{(3)},\phi^{(4)})$ corresponds to the first five terms on the RHS of \eqref{equ20}. This remainder term can be further rewritten as
\begin{equation*}
\begin{split}
\varepsilon^3(4\pi e\bar n)\{&\kappa\phi^{\varepsilon}_R +\kappa^2(\varepsilon\phi^{(1)}+\varepsilon^2\phi^{(2)})\phi^{\varepsilon}_R +\frac{\kappa^3}{2}\varepsilon^2(\phi^{(1)})^2\phi^{\varepsilon}_R\\ &-n^{\varepsilon}_R+\varepsilon^2\frac{\kappa^2}{2!}(\sqrt{\varepsilon}\phi^{\varepsilon}_R)^2 +\varepsilon^2\widehat R'({\varepsilon}\phi^{\varepsilon}_R)\},
\end{split}
\end{equation*}
for some $\widehat R'$ depending on ${\varepsilon}\phi^{\varepsilon}_R$. After divided by $\varepsilon^3$, the remainder equation for the Poisson equation is written as
\begin{equation*}
\begin{split}
\varepsilon\partial_x^2\phi^{\varepsilon}_R=&4\pi e\bar{n}\{\kappa\phi^{\varepsilon}_R+\varepsilon\kappa^2\phi^{(1)}\phi^{\varepsilon}_R -n^{\varepsilon}_R\}+\varepsilon^2{\mathcal R}_3,
\end{split}
\end{equation*}
where
\begin{equation}\label{R3}
\begin{split}
{\mathcal R}_3=[\frac{\kappa^2}{2!}({\varepsilon}\phi^{\varepsilon}_R) +\kappa^2(\phi^{(2)}+\frac{\kappa}{2}(\phi^{(1)})^2)]\phi^{\varepsilon}_R +\widehat R'({\varepsilon}\phi^{\varepsilon}_R).
\end{split}
\end{equation}

The other two equations \eqref{rem-1} and \eqref{rem-2} can be derived similarly and we omit the details.

\begin{lemma}\label{L8}
For $\alpha=0$, there exists some constant $C=C(\|\phi^{(i)}\|_{H^{\tilde s_i}},\varepsilon\|\phi^{\varepsilon}_R\|_{H^{1}})$ or when $\alpha=1,\cdots$ integers, there exists some constant $C=C(\|\phi^{(i)}\|_{H^{\tilde s_i}},\varepsilon\|\phi^{\varepsilon}_R\|_{H^{\alpha}})$ such that
\begin{equation}\label{equ29}
\begin{split}
&\|\mathcal R_3\|_{H^{\alpha}}\leq C(\|\phi^{(i)}\|_{H^{\tilde s_i}},\varepsilon\|\phi^{\varepsilon}_R\|_{H^{1}}) \|\phi^{\varepsilon}_R\|_{H^{\alpha}}, \ \ \ \alpha=0,1,\\
&\|\mathcal R_3\|_{H^{\alpha}}\leq C(\|\phi^{(i)}\|_{H^{\tilde s_i}},\varepsilon\|\phi^{\varepsilon}_R\|_{H^{\alpha}}) \|\phi^{\varepsilon}_R\|_{H^{\alpha}}, \ \ \ \forall\alpha\geq2.
\end{split}
\end{equation}
\begin{equation}\label{equ34}
\begin{split}
&\|\partial_t\mathcal R_3\|_{H^{\alpha}}\leq C(\|\phi^{(i)}\|_{H^{\tilde s_i}},\varepsilon\|\phi^{\varepsilon}_R\|_{H^{1}}) \|\partial_t\phi^{\varepsilon}_R\|_{H^{\alpha}}, \ \ \ \alpha=0,1,\\
&\|\partial_t\mathcal R_3\|_{H^{\alpha}}\leq C(\|\phi^{(i)}\|_{H^{\tilde s_i}},\varepsilon\|\phi^{\varepsilon}_R\|_{H^{\alpha}}) \|\partial_t\phi^{\varepsilon}_R\|_{H^{\alpha}}, \ \ \ \forall\alpha\geq2.
\end{split}
\end{equation}
\end{lemma}
Recalling the fact that $H^1$ is an algebra, the estimate for Lemma \ref{L8} is straightforward. The details are hence omitted. Combining Lemma \ref{L3} and Lemma \ref{L2}, we have
\begin{corollary}
For any $\alpha=0,1,2$, there exists some constant $C=C(\|\phi^{(i)}\|_{H^{\tilde s_i}},\varepsilon\|\phi^{\varepsilon}_R\|_{H^{2}})$ such that
\begin{equation*}
\begin{split}
\varepsilon\|\varepsilon\partial_t\partial_x^\alpha\mathcal R_3\|^2\leq C(\|\phi^{(i)}\|_{H^{\tilde s_i}},\varepsilon\|\phi^{\varepsilon}_R\|_{H^{2}}) |\!|\!|(u^{\varepsilon}_R,\phi^{\varepsilon}_R)|\!|\!|_{\varepsilon}^2.
\end{split}
\end{equation*}
\end{corollary}

\bigskip

\paragraph{\emph{Acknowledgment}} Y. Guo's research is supported in part by NSFC grant \#10828103 and NSF grant \#DMS-0905255. X. Pu's research is supported by NSFC grant \#11001285. Y. Guo thanks B. Pausader for his interests and discussions.

\end{CJK*}
\end{document}